%% file: heat_kernel_bounds.tex
\documentclass[a4paper,11pt,reqno]{article}
\usepackage{geometry}
\usepackage[normalem]{ulem}

\usepackage{amsthm}

\input{style_definitions}

\input{theorem_definitions}

\newtheorem{obs}[theorem]{\!\!}

\newcommand{\e}{{\rm e}}

\usepackage{times}
\usepackage{stmaryrd}

\input{math_definitions}
\usepackage[disable]{todonotes} 

\input{writing_definitions}

\usepackage{xcolor}

\newcommand{\br}{\bgroup\color{red}}
\newcommand{\er}{\egroup}

\newcommand{\bb}{\bgroup\color{blue}}
\newcommand{\eb}{\egroup}

\RenewEnviron{calc}{}  
\renewcommand{\cnewline}{}  
\renewcommand{\cand}{}  
\RenewEnviron{hide}{}  

\usepackage{tocloft}
\begin{document}

\input{title_arxiv_hkb}



\section{Introduction and main results}\label{sec-Intro}

In this paper we consider the stochastic differential equation on $\R^d$ given by 
\begin{align}
\label{eqn:diffusion_intro_equation}
\dd X_t = b(  t,   X_t) \dd t + \dd B_t, 
\end{align}
where $B$ is a Brownian motion and $b$ is a distribution of regularity $>-\frac12$. Such \emph{singular diffusions} (diffusions with distributional drift) appear as models for stochastic processes in random media (then $b$ would also be random, but independent of $B$), for example in~\cite{Br86, DeDi16, CaCh18}. They also appear as ``stochastic characteristics'' in Feynman-Kac type representations of singular SPDEs, for example in~\cite{GuPe17, CaCh18, KoPevZ}. In non-singular SPDEs, the stochastic characteristics would be formulated in terms of the Brownian motion, and they may be useful tools to infer information about the long-time behavior of the SPDE. For example, the asymptotic behavior of the total mass of the parabolic Anderson model  is typically derived via the Feynman-Kac formula~\cite{Ko16}, and for that purpose it is important that we understand the Brownian motion and its transition probabilities very well. When studying singular variants of the parabolic Anderson model, where the Brownian motion in the Feynman-Kac representation is replaced by a singular diffusion, we thus need to understand the transition probabilities of this singular diffusion. Moreover, since we are interested in the long-time behavior, we need quantitative control of the transition probabilities on arbitrarily long time intervals. This motivates our present work.

We show that the solution to~\eqref{eqn:diffusion_intro_equation} possesses a transition kernel $\Gamma_t : \R^d \times \R^d \rightarrow \R$ for all $t>0$. This means that under the measure $\P_x$ such that $X_0 =x$ we have for all $\phi\in C_{\rm b}(\R^d)$ 
\begin{align*}
\E_x[ \phi(X_t)] = \int_{\R^d} \phi(y) \Gamma_t(x,y)  \dd y. 
\end{align*}

The following theorem represents the main result of our paper, in which we show that the above transition kernel satisfies heat-kernel estimates. 
 
For any Banach space $\fX$ and $t>0$ we write $\|\cdot\|_{C_t \fX}$ for the norm on $C([0,t],\fX)$, which is defined for $f\in C([0,t],\fX)$ by
\begin{align*}
\| f\|_{C_t \fX} = \sup_{s\in [0,t]} \|f(s)\|_{\fX}. 
\end{align*}
$\Delta_{-1}b$ denotes the first Littlewood-Payley block and $\Delta_{\ge 0} b$ the sum of the positive Littlewood-Payley blocks (see Section~\ref{section:notation}). $B_{p,q}^s$ denotes a Besov space, see \cite{BaChDa11}. 

\begin{theorem}
\label{theorem:heat_kernel_bound}
Let  $\alpha \in (0, \frac12 )$ and $c > 1$.  
There exist a  $C > 1$ and a $\kappa \in (0,1)$ such that for all $b   = (b_t)_{t\ge 0}  \in  C([0,\infty), B_{\infty,1}^{-\alpha} (\R^d, \R^d) )   $, $\mu \in \N_0^d$ with $| \mu | \le 1$, and for all $t >0$, $x,y \in \R^d$:
\begin{align}
  \label{eqn:heat_kernel_upper_bound} 
  | \partial^{\mu}_x \Gamma_t (x, y) | 
  &\le C 
  \exp \left( C t  \Big[  \| \Delta_{-1} b \|_{  C_t   L^\infty}^{2}  + \| \Delta_{\ge 0} b \|_{  C_t   B_{\infty, 1}^{- \alpha}}^{\frac{2}{1 - \alpha}}  \Big]  \right) (t^{- \frac{| \mu |}{2}} \vee 1) p (c t, x - y) , \\
  \label{eqn:heat_kernel_lower_bound} 
  \Gamma_t (x, y)  &\ge 
  \frac{1}{C}  \exp \left( - C t  \Big[  \| \Delta_{-1} b \|_{  C_t   L^\infty}^{2}  + \| \Delta_{\ge 0} b \|_{  C_t   B_{\infty, 1}^{- \alpha}}^{\frac{2}{1 - \alpha}}  \Big]  \right)  p ( \kappa t , x - y) ,
\end{align}
where $p(t,x)=(2\pi t)^{-\frac{d}{2}}\e^{-|x|^2/2t}$ is the standard Gaussian kernel. 
\end{theorem}

As a corollary, we obtain the following estimate on the escape probability of the diffusion $X$ to leave a ball. 

\begin{corollary}
\label{corollary:escape_out_of_box_probability}
Let $\alpha \in (0,\frac12)$. There exists a $C>0$ such that for all  $b \in   C([0,\infty), B_{\infty,1}^{-\alpha}  (\R^d, \R^d))$, 
 $x\in \R^d$, $K>0$ and $T\ge 1$, and for $X$ solving \eqref{eqn:diffusion_intro_equation} with $\P_x(X_0=x)=1$: 
  \begin{align}
\notag 
&  \P_x  \bigg(\sup_{t\in [0,T]} | X_t - x |\ge K\bigg) \\
  \label{eqn:escape_out_of_box_probability}
&   \le C 
  \exp \Big( C T  \Big[  \| \Delta_{-1} b \|_{  C_T   L^\infty}^{2}  + \| \Delta_{\ge 0} b \|_{  C_T   B_{\infty, 1}^{- \alpha}}^{\frac{2}{1 - \alpha}}  \Big]  \Big) \exp \Big( - \frac{K^2}{C T} \Big)
     \end{align}
\end{corollary}

\begin{remark}
At least for constant $b$ the heat-kernel estimates are sharp: If $\lambda \in \R^d$ and $b=\lambda$, then $\Gamma_t(x,y) = p(t, y-x - \lambda t)$ and a simple computation shows that $\sup_{x\in \R^d} \frac{p(t,x-\lambda t)}{p(ct,x)} = c^{\frac{d}{2}} e^{\frac{1}{2(c-1)} t \lambda^2}$ and $\inf_{x\in \R^d} \frac{p(t,x-\lambda t)}{p(\kappa t,x)} = \kappa^{\frac{d}{2}} e^{- \frac{1}{2(1-\kappa)} t \lambda^2}$. Since in that case $\Delta_{\ge 0} b = 0$, this corresponds exactly to our bounds \eqref{eqn:heat_kernel_upper_bound} and \eqref{eqn:heat_kernel_lower_bound} (for $\mu=0$). 
\end{remark}
  
\begin{calc}
Indeed, if $X_t$ is the solution of 
\begin{align*}
\dd X_t = \lambda \dd t + \dd B_t,
\end{align*}
then $X_t = X_0 + \lambda t + B_t$ and thus for $\P$ being the probability such that $B_t$ is a standard Brownian motion and $\P_x$ the probability under which $X$ satisfies the SDE with  $X_0 = x$ a.s., 
\begin{align*}
\P_x(X_t \in A) 
= \P ( x+ \lambda t + B_t \in A) 
= \int_{A-x-\lambda t} p(t,y) \dd y 
= \int_{A} p(t,y-x-\lambda t) \dd y . 
\end{align*}
We have 
\begin{align*}
\frac{\Gamma_t(0,y)}{p(ct,y)}
= \frac{p(t,y-\lambda t)}{p(ct,y)}
= c^{\frac{d}{2}} \left[ \exp \left( \frac{|y|^2}{2ct} - \frac{|y-\lambda t|^2}{2t} \right) \right]. 
\end{align*}
We calculate
\begin{align*}
\sup_{y\in \R^d} |y|^2 - c |y -\lambda t|^2 . 
\end{align*}
The function over which we take the derivative is concave, so we calculate the point at which the gradient equals zero. 
The derivative equals $\partial_i |y|^2 - c |y -\lambda t|^2 = 2y_i - 2c  (y_i-\lambda_i t) $. 
This equals zero for $y_i = \frac{c\lambda_i t}{c-1}$ and thus 
\begin{align*}
\sup_{y\in \R^d} |y|^2 - c |y -\lambda t|^2 
= \frac{c-1}{c} |y|^2 
= \frac{c}{c-1} |\lambda|^2 t^2,
\end{align*}
and thus 
\begin{align*}
\sup_{y\in \R^d} \frac{\Gamma_t(0,y)}{p(ct,y)}
= c^{\frac{d}{2}} \exp \left( \frac{|\lambda|^2 t}{2(c-1)}   \right).
\end{align*}
For the lower bound we take instead of the supremum the infimum and replace $c$ by $\kappa$, which makes the function $|y|^2 - \kappa |y-\lambda t|^2$ convex. Then we find 
\begin{align*}
\inf_{y\in \R^d} \frac{\Gamma_t(0,y)}{p(\kappa t,y)}
= \kappa^{\frac{d}{2}} \exp \left( - \frac{|\lambda|^2 t}{2(1-\kappa)}   \right).
\end{align*}
\end{calc}  
  
\begin{remark}
\label{remark:inhomogeneous}
As we consider a time inhomogeneous drift, we could have also formulated the heat-kernel estimates for $\Gamma_{s,t}$ (with $0\le s< t$), which is the transition kernel from time $s$ to time $t$: If $\P_{s,x}$ is the probability measure under which $X_s =x$ and \eqref{eqn:diffusion_intro_equation} holds (for $t>s$), then 
$\E_{s,x}[ \varphi(X_t)] = \int_{\R^d} \varphi(y) \Gamma_{s,t}(x,y)  \dd y$.
However, to simplify notation we only consider the case $s=0$ and we write $\Gamma_t$ for $\Gamma_{0,t}$.
The heat-kernel estimates for $\Gamma_{s,t}$ follow by applying Theorem~\ref{theorem:heat_kernel_bound} with $b'_t = b_{t+s}$, $t \ge 0$.
\end{remark}

\subsection{Literature}
\label{subsection:literature}

 Diffusions with a distributional drift   were first considered by Bass and Chen~\cite{BaCh01} and Flandoli, Russo and Wolf~\cite{FlRuWo03}, both in the one-dimensional time-homogeneous setting. More recently, Delarue and Diel~\cite{DeDi16} used Hairer's rough path approach to singular SPDEs~\cite{Ha11, Ha13} to extend the results of~\cite{FlRuWo03} to the time-inhomogeneous case, and they applied this to construct a random directed polymer measure. Flandoli, Issoglio and Russo~\cite{FlIsRu17} were the first to consider multidimensional singular diffusions, but they require more regularity than in the previous works on the one-dimensional case (they consider the ``Young regime'', i.e., the distributional drift has regularity better than $-1/2$). 
Zhang and Zhao~\cite{ZhZh17} study the ergodicity and they derive heat-kernel estimates for singular diffusions in the Young regime. Cannizzaro and Chouk~\cite{CaCh18} use paracontrolled distributions to extend the approach of~\cite{DeDi16} to higher dimensions and the results of~\cite{FlIsRu17} to more singular drifts. They apply this to construct a random polymer measure that is closely related to the parabolic Anderson model.
 
In this paper we follow the approach of Cannizzaro and Chouk, although we restrict our attention to the more regular Young regime. This is crucial for our arguments.

As already mentioned, Zhang and Zhao~\cite{ZhZh17} also prove heat-kernel estimates for SDEs with distributional drifts in the Young regime. More precisely, they prove that there exist $c,C \ge 1$ such that for all $t\in (0,T]$ and $x, y\in \R^d$
\begin{align*}
\tfrac{1}{C} p(\tfrac{t}{c},x-y)\le  | \Gamma_t (x, y) | \le C p(ct,x-y). 
\end{align*}
Moreover, they give an upper bound on the gradient of the transition kernel, $\nabla \Gamma_t$. 
Here, the constant $C$ implicitly depends on $T$ and $\|b\|_{\cC^{-\alpha}}$.

If $b$ is the gradient of a function that does not dependent on time, then there is classical heat-kernel estimates for $\Gamma$, see for example Stroock~\cite[Theorem~4.3.9]{St08}. In that theorem we have $b = \nabla U$ for a smooth and bounded function $U$, but the estimate only depends on 
$\max U - \min U$, 
so by an approximation argument it extends to continuous and bounded $U$. This result is uniform in time, but also here the dependence of the constants on 
$\max U - \min U$ 
is implicit. 

In another work by the authors together with W. K\"onig \cite{KoPevZ}, our heat-kernel estimates are applied to derive the asymptotic behavior of the total mass of the parabolic Anderson model. In that application it is crucial to understand how the constant grows with $t$ and the norm of $b$. 
Therefore, we need our ``quantitative version'' of the heat-kernel estimates.

\subsection{Notation   and conventions  }
\label{section:notation}

We write $\N= \{1,2,\dots\}, \N_0= \{0\}\cup \N$ and $\N_{-1} = \{-1\} \cup \N_0$. 
  For the whole paper, $d$ is an element of $\N$ and will denote the dimension of the space.  
For families $(a_i)_{i \in \I}, (b_i)_{i\in \I}$ in $\R$ for an index set $\I$, we write $a_i \lesssim b_i$ to denote the existence of a $C>0$ such that $a_i \le C b_i$ for all $i\in \I$. 
We write $C_{\rm b}$ for the space of continuous bounded functions and $C_{\rm b}^\infty$ for the space of $C^\infty$ functions for which all their derivatives are bounded functions. 
We abbreviate function spaces and Besov spaces by omitting ``$(\R^d)$'' in the notation, for example we abbreviate $B_{p,q}^\beta(\R^d)$ to $B_{p,q}^\beta$. Moreover, we write $\cC^\beta$ for $B_{\infty,\infty}^\beta$ and $\cC_p^\beta$ for $B_{p,\infty}^\beta$.
We write $u \para v$  for the paraproduct between $u$ and $v$ (with the low frequencies of $u$ and the high frequencies of $v$), and $u\reso v$ for the resonance product; we adopt the notation from \cite{MaPe19} and refer to \cite{BaChDa11} as background material.
 
  In the rest of the paper   
$(\rho_i)_{i\in \N_{-1}}$ is a \emph{dyadic partition of unity}, meaning that $\rho_{-1}$ is supported in a ball around $0$, $\rho_0$ is supported in an annulus, $\rho_i (x) = \rho_0(2^{-i} x)$ for $i \in \N_0$, $\sum_{i \in \N_{-1}} \rho_i =\1$, $\frac12 \le \sum_{i \in \N_{-1}} \rho_i^2 \le 1$ and $\supp \rho_i \cap \supp \rho_j = \emptyset$ if $|i-j| \ge 2$. 
For $i \in \N_{-1}$ we write $\Delta_i$ for the corresponding Littlewood-Payley blocks ($\cF$ denotes the Fourier transform)
\begin{align*}
\Delta_i f = \rho_i (\rD) f = \cF^{-1}( \rho_i \cF(f)) = \cF^{-1}(\rho_i) *f . 
\end{align*}
Moreover, we define $\Delta_{\ge 0} f $ to be the sum of all the positive Littlewood-Payley blocks: 
\begin{align*}
\Delta_{\ge 0} f = \sum_{i\in \N_0} \Delta_i f. 
\end{align*}

\section{Diffusions with distributional drift and their heat-kernel estimates}
\label{section:heat_kernel_bounds}

\textbf{Throughout this section we fix $T>0$.}
Let $\alpha \in (0,\frac12)$. 
For $b\in  
  C( [0,T] , B_{\infty,1}^{-\alpha}(\R^d,\R^d))  $
 we consider the stochastic differential equation 
\begin{align}
\label{eqn:diffusion_with_drift_b_in_section}
\dd X_t = b(  t,   X_t) \dd t + \dd B_t. 
\end{align}
  For $t >0$ let $\cL_t$   be the operator 
\begin{align}
\label{eqn:def_operator_cL}
 \cL_t  
 = \tfrac12 \Delta +   b_{t}   
  \cdot \nabla. 
\end{align}
We consider the following Cauchy problem for $u\colon   [0,T]   \times \R^d \rightarrow \R$ 
with   terminal   condition $\phi$:
\begin{align}
\label{eqn:cauchy_problem}
\begin{cases}
\partial_t u     + \cL_{t}   
 u = 0 & \qquad \mbox{ on }  [0,T)  \times \R^d, \\
  u(T,\cdot)   = \phi  & \qquad \mbox{ on } \R^d. 
\end{cases}
\end{align}

The solution theory for the Cauchy problem will be given in Proposition~\ref{proposition:solution_cauchy_problem_in_cC}.
We write $u^\phi$ for the solution to \eqref{eqn:cauchy_problem}. 
But let us first discuss how to interpret \eqref{eqn:diffusion_with_drift_b_in_section} in terms of a martingale problem.

\begin{definition}\label{definition:martingale_problem}
We say that a stochastic process $X= (X_t)_{t\in  [0,T] }$ on a probability space $(\Omega,\P)$ is a {\em solution to the SDE \eqref{eqn:diffusion_with_drift_b_in_section} on $[0,T]$    with initial condition $X_0=x$}  if it satisfies the martingale problem for $(  (\cL_t)_{t\in (0,T]}  , \delta_x)$, i.e., if $\P( X_0 = x) =1$ 
and for all $f \in C([0,T],L^\infty(\R^d))$, all $\phi\in C_{\rm c}^\infty(\R^d)$ and for $u=u^\phi$ being the solution to the Cauchy problem \eqref{eqn:cauchy_problem}, 
 the process 
\begin{align*}
\Big(u(  t  ,X_t)   - \int_0^t f(s,X_s) \dd s   \Big)_{t\in [0,T]}
\end{align*}  
is a martingale. 
\end{definition}

The martingale problem has a unique solution:

\begin{theorem} \textnormal{\cite[Theorem 1.2]{CaCh18}}
\label{theorem:martingale_solution}
Let $\alpha \in (0,\frac12)$.
For all $x\in \R^d$ and $b\in   C([0,T],\cC^{-\alpha}(\R^d,\R^d))   
$ there exists a unique solution to the martingale problem for 
$(  (\cL_t)_{t\in (0,T]}  , \delta_x)$, in the sense that there is a unique probability measure $\P_x$ on $\Omega = C( [0,T] ,\R^d)$ such that the coordinate process $X_t(\omega) = \omega(t)$ satisfies the martingale problem for $(  (\cL_t)_{t\in (0,T]}  , \delta_x)$. Moreover, $X$ is a strong Markov process under $\P_x$   and the measure $\P_x$ depends (weakly) continuously on the drift $b$.   
\end{theorem}

\begin{remark}
\label{remark:continuity_drift_and_prob_with_other_starting_time}
The continuity of the solution $\P$ in terms of the drift is not mentioned in \cite[Theorem 1.2]{CaCh18}, but it can be extracted from their proof. 

Observe that Theorem~\ref{theorem:martingale_solution} also implies that there exists a unique probability measure $\P_{s,x}$ on $C([s,T],\R^d)$ such that the coordinate process satisfies the martingale problem for $((\cL_t)_{t\in (s,T]},\delta_x)$. 
This can be obtained by applying Theorem~\ref{theorem:martingale_solution} to a shift of the drift, as is mentioned in Remark~\ref{remark:inhomogeneous}. 
\end{remark}

Next, our aim is to show that $X$ admits a transition density $\Gamma_{s,t}$ for $0\le s<t\le T$ (Proposition~\ref{proposition:transition_density}), which means that for $\varphi\in C_{\rm c}(\R^d)$ and $x\in \R^d$ and with $\P_{s,x}$ as in Remark~\ref{remark:continuity_drift_and_prob_with_other_starting_time}
\begin{align}
\label{eqn:transition_kernel_property}
\E_{s,x}[ \varphi(X_t)] = \int_{\R^d} \varphi(y) \Gamma_{s,t}(x,y) \dd y. 
\end{align} 
We do this by showing that $\Gamma_{t,T}(x,y) = u^{\delta_y}(  t  ,x)$ for the solution $u^{\delta_y}$ to~\eqref{eqn:cauchy_problem} with   terminal   condition $u(T,\cdot) = \delta_y$.

In order to construct the solution $u^{\delta_y}$ we have to slightly extend the results of \cite{CaCh18}. 
Indeed, in \cite[Theorem 3.1 and 3.2]{CaCh18} the well-posedness of the Cauchy problem is shown for $\phi \in \cC^{\beta}$ with $\beta \in (1+\alpha, 2-\alpha)$, and $\delta_z$ is not in this space. 
  The solution theory in \cite{CaCh18} is formulated in terms of mild solutions:  A \emph{mild solution} of \eqref{eqn:cauchy_problem} is a fixed point $u$ of $\Phi$, i.e., $\Phi u =u$, where $\Phi$ is defined on $C([0,T],\cS')\cap [\bigcup_{p\in [1,\infty]} C([0,T), \cC_p^\beta(\R^d))]$ for $\beta>1+\alpha$ by 
\begin{align}
\label{eqn:map_for_cauchy_fixed} 
(\Phi u)_s =  P_{T-s} \phi  -  \int_s^T P_{r-s} (  b_{r}   
 \cdot \nabla u_r) \dd r,   
\end{align}
where $P_t \phi := p(t,\cdot) * \phi$ for $t>0$ and $P_0 \phi = \phi$ (that $\Phi$ is well-defined follows by \ref{obs:regularities_product}). 
\begin{calc}
Let us check this by doing formal calculations. 
This means we treat $b$ as a nice function (say in $C_{\rm b}^\infty$). 
The calculations are supported by Lemma~\ref{lemma:derivative_of_integral_parameter_and_integrand_same}. 
If $\Phi u = u$, 
$F(r) = \int_0^s f(s-r,r) \dd r $, $f(s,r) = P_s(b_r \cdot \nabla u_r)$,
then 
\begin{align*}
\partial_s u_s 
& = - \Delta P_{T-s} \phi - P_{0}  ( b_s \cdot \nabla u_s) + \int_s^T \Delta P_{r-s} (b_r \cdot \nabla u_r) \dd r \\
& = - \Delta u_s - b \cdot \nabla u_s = - \cL_s u, 
\end{align*}
so that $u$ is a solution to \eqref{eqn:cauchy_problem}.
\end{calc}
In order to allow $\delta_y$ as a terminal condition, we will consider a different space that ``allows a blowup as $t \uparrow T$''. 
However, for notational elegance, we instead consider a space with ``a blowup at $0$'' and mention that $u$ is a fixed point of $\Phi$ if and only if $v$ given by $v(t,\cdot) = u(T-t,\cdot)$ is a fixed point of $\Theta$, given by 
\begin{align}
\label{eqn:def_Psi}
(\Theta v)_s =  P_s \phi  +  \int_{0}^s P_{s-r} (  b_{T-r}   
 \cdot \nabla v_r) \dd r, 
\end{align}
\begin{calc}
Sanity check, if $u = \Phi u$, then 
\begin{align*}
v_{s} = u_{T-s}
= (\Phi u)_{T-s}
& = P_{s} \phi  -  \int_{T-s}^T P_{r-(T-s)} ( b_{r} 
 \cdot \nabla u_r) \dd r \\
& = P_{s} \phi  -  \int_0^{s} P_{s-q} ( b_{T-q} 
 \cdot \nabla u_{T-q}) \dd q \\
 & = P_{s} \phi  -  \int_0^{s} P_{s-r} ( b_{T-r} 
 \cdot \nabla v_{r}) \dd r . 
\end{align*}
\end{calc}
so that we call $v$ a \emph{mild solution} of 
\begin{align}
\label{eqn:cauchy_problem_dual}
\begin{cases}
\partial_t v     - \cL_{T-t}   
 v = 0 & \qquad \mbox{ on }  (0,T]  \times \R^d, \\
  v(0,\cdot)   = \phi  & \qquad \mbox{ on } \R^d. 
\end{cases}
\end{align}
\begin{calc}
Indeed, let us suppose that $u$ is a solution to \eqref{eqn:cauchy_problem} and $v(t,\cdot) = u(T-t,\cdot)$. 
Then $v(0,x) = u(T,x) = \phi$ and because $\nabla v(s, \cdot) = \nabla u(T-s, \cdot)$, 
\begin{align*}
\partial_t v(t,x) \
= - \partial_t u(T-t,x) 
= \cL_{T-t} u(T-t,x) 
= \cL_{T-t} v(t,x). 
\end{align*}
\end{calc}
We will show that $\Theta$ has a fixed point in the following space (for suitable $\delta, \beta$). 
For $\delta\ge 0$, $\beta \in \R$ and $t>0$ we define
\begin{align*}
\|u \|_{M^{\delta}_t \cC_p^\beta} &  = \sup_{s\in (0,t]} s^{  \delta} \| u_s \|_{\cC_p^\beta}, \\
M^{\delta}_t \cC_p^\beta
& = \{ u \in C( (0,t] , \cC_p^\beta) : \|u \|_{M^{\delta}_t \cC_p^\beta} <\infty\}. 
\end{align*}

The following proposition is a slight extension of \cite[Theorem 3.1 and 3.2]{CaCh18}. 

\begin{proposition}
\label{proposition:solution_cauchy_problem_in_cC}
Let $\alpha \in (0,\frac12)$, $p\in [1,\infty]$ and $\gamma >  \alpha-1 $. 
For $\phi \in \cC_p^\gamma$, $b\in   C( [0,T] ,B_{\infty,1}^{-\alpha})$, $\beta \in (1+\alpha, 2 - \alpha)$ and $\epsilon>0$ the Cauchy problem \eqref{eqn:cauchy_problem} has a unique mild solution $u^{\phi,b}$ in $C([0,T],\cC_p^{(\gamma -\epsilon) \wedge \beta}) \cap C([0,T),\cC_p^{\beta})$ such that $u^{\phi,b}(t) \in \cC^\beta$ for all $t\in [0,T)$. 
Moreover, for all $t>0$ the map $\cC_p^\gamma \times   C( [0,T] ,B_{\infty,1}^{-\alpha})   
\rightarrow \cC^\beta$ given by $(\phi,b) \mapsto u^{\phi,b}(t,\cdot)$ is locally Lipschitz. 
\end{proposition}

Another difference with \cite{CaCh18} is that we consider $b\in   C([0,T],B_{\infty,1}^{-\alpha})   
$ instead of $b\in   C([0,T],\cC^{-\alpha})   
$. Since $B_{\infty,1}^{-\alpha} \subset \cC^{-\alpha} \subset B_{\infty,1}^{-\alpha -\epsilon}$ (as continuous embeddings),  
this does not make much of a difference. 
But our heat-kernel estimates depend on the $B_{\infty,1}^{-\alpha}$-norm and for their derivation it is more convenient to work with $B_{\infty,1}^{-\alpha}$.


Before we prove Proposition~\ref{proposition:solution_cauchy_problem_in_cC} we present two auxiliary facts, Lemma~\ref{lemma:bound_P_t} and \ref{obs:regularities_product}.

  We write $B$ for the beta function (see e.g. \cite[Section 1.1]{AnAsRo99}),
which is the function $ B:  (0,\infty)^2 \rightarrow (0,\infty)$ given by
\begin{align}
\label{eqn:beta_function_def}
B(\beta, \gamma) = \int_0^1 r^{\gamma-1} (1-r)^{\beta -1} \dd r .
\end{align}

\begin{lemma}
\label{lemma:bound_P_t}
Let $p \in [1,\infty]$, $\kappa\ge 0$, $\delta \in [0,1)$, $\alpha,\gamma \in \R$ and $\beta \in  [-\alpha,2-\alpha)  $. 
 
There exists a $C>0$ such that for all $t \in (0,1]$, 
\begin{align}
\label{eqn:bounds_P_t_and_time_integral_P_t}
& \| s \mapsto P_s \phi \|_{M^{\frac{\kappa}{2}}_t \cC^{\gamma+\kappa}_p} \le C \| \phi \|_{\cC^\gamma_p}, 
\qquad  \|P_t \phi - \phi \|_{\cC_p^{\gamma-2\delta}} \le C t^{\delta} \|\phi\|_{\cC_p^\gamma}, \\
\label{eqn:bound_time_integral_P_t_gamma}
& \left\| s \mapsto \int_0^s P_{s-r} w_r \dd r  \right\|_{M^\delta_t \cC^{\beta}_p} \le C t^{1 - \frac{\alpha + \beta}{2}}  \| w \|_{M^\delta_t \cC^{-\alpha}_p}. 
\end{align}
\end{lemma}
\begin{proof}
In \cite[Lemma A.7]{GuImPe15} it is proven (for $p=\infty$, but can be caried on mutatis mutandis for general $p \in [1,\infty]$\begin{calc} \ 
by applying \cite[Lemma 2.2]{BaChDa11} with $p$ instead of $\infty$ at the right place\end{calc}) 
 that
for all $\kappa \ge 0$ and $\gamma \in \R$ there exists a $C>0$ such that for all $t\in (0,1]$
\begin{align}
\label{eqn:bound_on_heat_kernel}
\| P_t \phi \|_{\cC_p^{\gamma+ \kappa}} \le  C t^{- \frac{\kappa}{2} }  \|\phi \|_{\cC_p^\gamma}, 
\end{align}
which implies the first bound in \eqref{eqn:bounds_P_t_and_time_integral_P_t}. 
The second bound in \eqref{eqn:bound_time_integral_P_t_gamma} follows by \eqref{eqn:bound_on_heat_kernel} as 
\begin{align*}
\|P_t \phi - \phi\|_{\cC_p^{\gamma -2 \delta}}
& = \|\int_0^t \partial_s P_s \phi \dd s \|_{\cC_p^{\gamma -2 \delta}}
\le \int_0^t \| P_s \Delta \phi \|_{\cC_p^{\gamma -2 \delta}} \dd s  \\
& \lesssim \int_0^t s^{-\frac{2-2\delta}{2}} \dd s  \|\Delta \phi \|_{\cC_p^{\gamma -2 }} \lesssim t^\delta \|\phi\|_{\cC_p^\gamma}. 
\end{align*}
The bound in \eqref{eqn:bound_time_integral_P_t_gamma} is also proven in \cite[Lemma A.9]{GuImPe15}, we give the proof to be self-contained. 
By applying \eqref{eqn:bound_on_heat_kernel} \begin{calc}(with $\gamma = -\alpha$ and $\kappa = \alpha+ \beta $ which is positive by assumption)\end{calc} we obtain  for $t \in (0,1]$ \begin{calc} and $s\in [0,t]$ 
\begin{align*}
\left\|  P_{t-s} w_s  \right\|_{\cC_p^\beta}
\cand \begin{calc} 
\lesssim (t-s)^{- \frac{\alpha+ \beta}{2}} 
\left\|   w_s  \right\|_{\cC_p^{-\alpha} }
\end{calc} \cnewline
& \lesssim (t-s)^{- \frac{\alpha+ \beta}{2}} s^{- \delta }
 \| w \|_{M_{T}^{\delta }\cC_p^{-\alpha} }, 
\end{align*}
and thus 
\end{calc}
\begin{align}
 \left\| \int_0^t P_{t-s} w_s \dd s \right\|_{\cC_p^\beta} 
& \notag 
 \lesssim 
\int_0^t (t-s)^{- \frac{\alpha+ \beta}{2}} s^{- \delta } \dd s
 \| w  \|_{M_{T}^{\delta}\cC_p^{-\alpha} } \\
 & \lesssim 
 \label{eqn:integral_P_b_dot_nabla_v_bound}
t^{-\delta + 1 - \frac{\alpha+\beta}{2}} 
B\left(1 - \tfrac{\alpha + \beta}{2}, 1 - \delta  \right)
 \| w  \|_{M_{T}^{\delta}\cC_p^{-\alpha} } . 
\end{align}
\begin{calc}
Observe that $1 - \frac{\alpha + \beta}{2}, 1-\delta \in (0,\infty)$ by assumption. 
\end{calc}
This proves \eqref{eqn:bound_time_integral_P_t_gamma}.
\end{proof}

\begin{obs}
\label{obs:regularities_product}
Let $\alpha > 0$ and let $\beta >  1 + \alpha$ and $\epsilon > 0$ be such that $1 + \alpha + \epsilon \le \beta$. Then we have by Theorem~\ref{theorem:product_estimates} together with Bernstein's inequality (\cite[Lemma 2.1 or 2.78]{BaChDa11}):
\begin{align*}
\| a \cdot \nabla w \|_{B_{p,\infty}^{-\alpha}} \lesssim 
\|a\|_{B_{\infty,1}^{-\alpha}} \| \nabla w\|_{B_{p,\infty}^{\alpha+\epsilon}} \lesssim \|a\|_{B_{\infty,1}^{-\alpha}} \| w\|_{B_{p,\infty}^{\beta}}.
\end{align*}
\end{obs}

\begin{proof}[Proof of Proposition~\ref{proposition:solution_cauchy_problem_in_cC}]
If $\gamma \ge \beta$, then the statement follows directly from \cite[Theorem 3.2]{CaCh18}. Therefore, we assume that $\gamma<\beta$ and it is sufficient to  show that the statement holds for ``$t_0$'' instead of ``$T$'', where $t_0$ will be chosen small, as we can extend the solution to $[t_0,T]$ by \cite[Theorem 3.2]{CaCh18}.

As mentioned before, it is sufficient to consider the fixed point problem for $\Theta$ as in \eqref{eqn:def_Psi} instead of $\Phi$. 
Let us write $\Theta_t^\phi$ for $\Theta$ as in \eqref{eqn:def_Psi} but with ``$T$'' replaced by ``$t$''.
We will show that there exists a $t_0$ such that 
\begin{enumerate}
\item 
\label{item:fixed_point_in_M}
$\Theta_{t_0}^\phi$ has a unique fixed point in $M_{t_0}^{\frac{\beta-\gamma}{2}}\cC_p^\beta$, 
\item 
\label{item:fixed_point_in_C}
$\Theta_{t_0}^\phi$ has a unique fixed point in $C([0,t_0],\cC_p^{\gamma-\epsilon})$, 
\item 
\label{item:fixed_points_agree}
$\Theta_{t_0}^\phi$ maps $C((0,t_0],\cC_p^{\beta})$ and thus $M_{t_0}^{\frac{\beta-\gamma}{2}}\cC_p^\beta$ into $C([0,t_0],\cC_p^{\gamma-\epsilon})$, so that the fixed point in $C([0,t_0],\cC_p^{\gamma-\epsilon})$ agrees with the fixed point in $M_{t_0}^{\frac{\beta-\gamma}{2}}\cC_p^\beta$, 
\item 
\label{item:continuity_at_t_for_p}
for all $t>0$ the map $\cC_p^\gamma \times   C( [0,T] ,B_{\infty,1}^{-\alpha})   
\rightarrow \cC_p^\beta$ given by $(\phi,b) \mapsto u^{\phi,b}(t,\cdot)$ is locally Lipschitz,
\item 
\label{item:increased_integrability}
the fixed point $v$ satisfies $v(t) \in \cC^\beta$ for $t\in (0,t_0]$ and the continuity in \ref{item:continuity_at_t_for_p} can be shown for $p=\infty$, by showing that we can ``increase the integrability parameter $p$ to $\infty$''. 
\end{enumerate}

First, we assume that $\gamma>-\alpha$ and show \ref{item:fixed_point_in_M}--\ref{item:increased_integrability}. After that we show how one can treat $\gamma \in (\alpha-1,-\alpha]$ too. 

\ref{item:fixed_point_in_M}
By combining the observation in~\ref{obs:regularities_product} with Lemma~\ref{lemma:bound_P_t} with $\kappa = \beta -\gamma $ and $\delta = \frac{\beta - \gamma}{2}$ (observe that by assumption $\kappa>0$ and $\delta \in (0,1)$, because $0<\beta - \gamma < 2 - \alpha + \alpha$); for $t\in (0,1]$
\begin{align}
\| \Theta_t^\phi v \|_{M^{\frac{\beta-\gamma}{2}}_{t} \cC_p^\beta} 
\notag & \lesssim \|\phi \|_{\cC_p^\gamma}
+
\| s \mapsto \int_0^s P_{s-r}( b_{t-r} \cdot \nabla v_r) \dd r \|_{M_{t}^{\frac{\beta - \gamma}{2}} \cC_p^\beta} \\
\notag & \lesssim \|\phi \|_{\cC_p^\gamma}
+ t^{1- \frac{\alpha +\beta}{2}} 
\| s \mapsto  b_{t-s} \cdot \nabla v_s \|_{M_{t}^{\frac{\beta - \gamma}{2}} \cC_p^{-\alpha}} \\
\label{eqn:bound_on_Theta_t_phi}
& \lesssim \|\phi \|_{\cC_p^\gamma} + t^{1- \frac{\alpha +\beta}{2}} 
\|b\|_{C_1 B_{\infty,1}^{-\alpha} }
 \left\|  v  \right\|_{M^{\frac{\beta-\gamma}{2}}_t \cC_p^{\beta} },
\end{align} 
and, moreover 
\begin{align}\label{eqn:contraction_bound_T}
\| \Theta_t^\phi v - \Theta_t^\phi \tilde v \|_{M_{t}^{\frac{\beta - \gamma}{2}} \cC_p^\beta} 
& \lesssim t^{1 - \frac{\alpha + \beta}{2} } \|b\|_{  C_1   B_{\infty,1}^{-\alpha}} \| v - \tilde v  \|_{M_{t}^{\frac{\beta- \gamma}{2}} \cC_p^\beta}. 
\end{align}
That $\Theta_t^\phi v$ forms an element of $C((0,t],\cC_p^\beta)$ follows by Lemma~\ref{lemma:bound_P_t}. 
Therefore, with \eqref{eqn:bound_on_Theta_t_phi} it follows that $\Theta_t^\phi$ maps $M^{\frac{\beta-\gamma}{2}}_t \cC^\beta_p$ to itself. 
By \eqref{eqn:contraction_bound_T} then follows that for sufficiently small $t_0$ the map $\Theta_{t_0}^\phi$ is a contraction on the Banach space $M^{\frac{\beta-\gamma}{2}}_{t_0} \cC^\beta_p$ and it has a unique fixed point in that space. 

\ref{item:fixed_point_in_C} When $t_0$ is as above, then $\Theta_{t_0}^\phi$ has a unique fixed point in $C([0,t_0],\cC_p^{\gamma-\epsilon})$ which follows from the following estimates which follow similarly as the above ones (use that $\gamma > \beta$ and \eqref{eqn:bound_time_integral_P_t_gamma} with $\beta = -\alpha$ and $\delta=0$) 
\begin{align*}
\| \Theta_t^\phi v \|_{C([0,t],\cC_p^{\gamma-\epsilon})} 
& \lesssim \|\phi \|_{\cC_p^{\gamma-\epsilon}} + t 
\|b\|_{C_1 B_{\infty,1}^{-\alpha} }
 \left\|  v  \right\|_{C([0,t],\cC_p^{\gamma-\epsilon})}, \\
\| \Theta_t^\phi v - \Theta_t^\phi \tilde v \|_{C([0,t],\cC_p^{\gamma-\epsilon})} 
\cand 
\begin{calc}
= \| s \mapsto \int_0^s P_{s-r}( b \cdot \nabla (v_r - \tilde v_r)) \dd r \|_{C([0,t],\cC_p^{\gamma-\epsilon})} 
\end{calc}
\cnewline
\cand 
\begin{calc}
\lesssim t \| s \mapsto  b \cdot \nabla (v_s - \tilde v_s)\|_{C([0,t],\cC_p^{-\alpha})} 
\end{calc} \cnewline 
& \lesssim t \|b\|_{  C_1   B_{\infty,1}^{-\alpha}} \| v - \tilde v  \|_{C([0,t],\cC_p^{\gamma-\epsilon})} .
\end{align*}

\ref{item:fixed_points_agree}
That $\Theta_{t_0}^\phi$ maps $C((0,t_0], \cC_p^\beta)$ into $C([0,t_0], \cC_p^{\gamma-\epsilon}) $, which means that $\Theta_{t_0}^\phi (v) (t)$ converges to $\phi$ in $\cC_p^{\gamma-\epsilon}$ as $t\downarrow 0$ for any $v\in C((0,t_0], \cC_p^\beta)$,  follows from the second bound in \eqref{eqn:bounds_P_t_and_time_integral_P_t} and the following estimate  (by \ref{obs:regularities_product}, which follows similarly to \eqref{eqn:bound_on_Theta_t_phi})
\begin{align*}
\|\Theta_{t_0}^\phi(v)(t) - \phi \|_{\cC_p^{\gamma-\epsilon}}
\le \| P_t \phi - \phi\|_{\cC_p^{\gamma-\epsilon}} 
+ t^{1- \frac{\alpha +\beta}{2}} 
\|b\|_{C_1 B_{\infty,1}^{-\alpha} }
 \left\| v  \right\|_{C((0,t_0], \cC_p^{\beta}) }. 
\end{align*}

\ref{item:continuity_at_t_for_p} 
Let us write $v^{\phi,b}$ for the solution of \eqref{eqn:cauchy_problem_dual} (with $\cL_t$ as in \eqref{eqn:def_operator_cL}). 
To see the continuity of the solution with respect to $b$ and $\phi$, let $b_1, b_2 \in   C([0,t_0], B_{\infty,1}^{-\alpha})  $ and $\phi_1,\phi_2 \in \cC_p^{\gamma}$. 
Let $v_i = v^{\phi_i,b_i}$ for $i\in \{1,2\}$. By Lemma~\ref{lemma:bound_P_t} and by~\ref{obs:regularities_product} we have 
\begin{align*}
\|v_1 - v_2\|_{M_{t}^{\frac{\beta - \gamma}{2}} \cC_p^\beta}
& \lesssim 
\|\phi_1 - \phi_2 \|_{\cC_p^\gamma} 
+ t^{ 1 - \frac{\alpha + \beta }{2}} 
\|b_1\|_{  C_t   B_{\infty,1}^{-\alpha} }
 \left\|  v_1 - v_2  \right\|_{M^{\frac{\beta-\gamma}{2}}_t \cC_p^{\beta} } \\
 & \qquad 
 + t^{ 1 - \frac{\alpha + \beta}{2}} 
\|b_1 - b_2\|_{  C_t   B_{\infty,1}^{-\alpha} }
 \left\|  v_2  \right\|_{M^{\frac{\beta-\gamma}{2}}_t \cC_p^{\beta} }. 
\end{align*}
Hence there exists a $\delta \in (0,t_0)$ (small enough, e.g., $\delta^{ 1 - \frac{\alpha + \beta}{2}}  \|b_1\|_{  C_\delta   B_{\infty,1}^{-\alpha} } <\frac12 $) such that 
\begin{align}
\label{eqn:continuity_estimate_small_delta}
\|v_1 - v_2\|_{M_{\delta}^{\frac{\beta - \gamma}{2}} \cC_p^\beta}
\lesssim 
\|\phi_1 - \phi_2 \|_{\cC_p^\gamma} 
+ 
\|b_1 - b_2\|_{  C_{t_0}  B_{\infty,1}^{-\alpha} }
 \left\|  v_2  \right\|_{M^{\frac{\beta-\gamma}{2}}_{t_0} \cC_p^{\beta} }. 
\end{align}
So for $t\in (0,\delta]$ we obtain the desired continuity. 
By an iteration argument we can obtain the continuity for all $t\in (0,t_0]$, as for example for $t\in (\delta,2\delta]$ we have $v_i(t) = v^{v_i(\delta),b}(t-\delta)$. 

\ref{item:increased_integrability}
It remains to show that we can increase the integrability from $p$ to $\infty$, i.e., that $v_t \in \cC^\beta$ for all $t > 0$ and that also as an element of $\cC^\beta$ the solution $v_t$ for fixed $t>0$ depends continuously on $b$ and $\phi$. First we show that if $t > 0$, then $v_s \in \cC^\beta$ for all $s > t$. To simplify notation we only consider the most extreme case $p=1$, but the argument for general $p$ is essentially the same. Let $n\in \N_0$ be such that 
\begin{align*}
n (\beta - \gamma) < d , \qquad (n+1) (\beta - \gamma) \ge d.
\end{align*}
Write $p_0=1$ and for $i\in \{1,\dots, n\}$
\begin{align*}
p_i = \frac{d}{d - i(\beta - \gamma)} \in (1,\infty). 
\end{align*}
Then $\beta - \frac{d}{p_{n}} \ge \gamma$ and $\beta - d ( \tfrac{1}{p_{i-1}} - \tfrac{1}{p_{i}}) = \gamma$ for all $i \in \{1,\dots, n-1 \}$, 
\begin{calc}
indeed
\begin{align*}
\frac{d}{p_n} & = d- n (\beta - \gamma) 
\begin{cases}
>0 \\
= d - (n+1)(\beta-\gamma) + \beta - \gamma \le \beta - \gamma, 
\end{cases}
\\
 d ( \tfrac{1}{p_{i}} - \tfrac{1}{p_{i-1}})
&  = d - i(\beta -\gamma) - (d - (i-1)(\beta - \gamma)) = \gamma - \beta,
\end{align*}
\end{calc}
hence the Besov embedding theorem \cite[Proposition 2.71]{BaChDa11} gives $\cC_{p_{i-1}}^\beta \subset \cC_{p_i}^\gamma$ for all  $i \in \{1,\dots,n-1\}$, and $\cC_{p_n}^\beta \subset \cC^\gamma$. 
We have $v_{\frac{t}{n}} \in \cC_1^\beta \subset \cC_{p_1}^{\gamma}$. 
By considering the equation~\eqref{eqn:cauchy_problem_dual} with initial condition $v_{\frac{t}{n}}$ we obtain that $v_s$ is in $\cC_{p_1}^\beta$ for $s> \frac{t}{n}$, in particular $v_{\frac{2}{n}t} \in \cC_{p_2}^\gamma $. Repeating the argument we obtain $v_{\frac{i}{n}t} \in \cC_{p_i}^\gamma$ for all $i\in \{1,\dots,n-1\}$ and $v_t\in \cC^\beta$, so indeed $v_s \in \cC^\beta$ for all $s > t$. 
As $t$ was arbitrary, we have shown that $v_t \in \cC^\beta$ for all $t > 0$. 
As all the inclusions $\subset$ above are given by continuous embeddings, the continuity of the solution with respect to $\phi$ and $b$ follows from the continuity shown in \ref{item:continuity_at_t_for_p}. 

We are left to show that we can also treat $\gamma\in (\alpha - 1, -\alpha]$. Let $\gamma$ be as such. We choose $\tilde \beta \in ( 1+\alpha , 2- \alpha)$ such that $\tilde \beta - \gamma <2$. Then we have $\tilde \beta> \gamma$ and $\frac{\tilde \beta- \gamma}{2}\in (0,1)$, so that the conditions of observation~\ref{obs:regularities_product} and  Lemma~\ref{lemma:bound_P_t} are satisfied. Hence we obtain also \eqref{eqn:bound_on_Theta_t_phi} and \eqref{eqn:contraction_bound_T} with ``$\tilde \beta$'' instead of ``$\beta$''. So then we find a $\tilde t_0 \in (0,t_0)$ such that $\Theta_{\tilde t_0}^\phi$ has a fixed point $\tilde v$ in $M_{\tilde t_0}^{\frac{\tilde \beta - \gamma}{2}} \cC_p^{\tilde \beta}$. 
Let $w$ be the fixed point of $\Phi_{t_0-\tilde t_0}^{\tilde v(\tilde t_0)}$ in $M_{t_0-\tilde t_0}^{\frac{\beta-\tilde\beta}{2}}$, which exists by \ref{item:fixed_point_in_M} because $\tilde \beta > - \alpha$. 
Then $v(t) := \tilde v(t)$ for $t\in (0,\tilde t_0]$ and $v(t):= w(t-t_0)$ for $t\in (\tilde t_0,t_0]$ is a fixed point of $\Phi_{t_0}^\phi$ such that $(\tilde t_0,t_0] \mapsto \cC_p^\beta$, $t\mapsto v(t)$ is continuous. As $\tilde t_0$ can be taken arbitrarily small, we conclude that $v\in C((0,t_0],\cC_p^\beta)$. 
Similarly, we can obtain the continuity of the solution by using \eqref{eqn:continuity_estimate_small_delta} with ``$(\beta,\gamma)$'' replaced by ``$(\beta,\tilde \beta)$'' and using \eqref{eqn:continuity_estimate_small_delta} with ``$(\beta,\gamma)$'' replaced by ``$(\tilde \beta,\gamma)$''. 
\end{proof}

\begin{obs}
A direct computation using that $\Delta_i \delta_z(x) = \cF^{-1}(\rho(2^{-i} \cdot))(x-z) = 2^{id} \cF^{-1}(\rho)(2^i (x-z)) $ for $i \ge 0$ shows that the Dirac delta $\delta_z$ is in $\cC_p^{-d(1-\frac{1}{p})}$ for all $p \in [1,\infty]$, so in particular $\delta_z \in \cC_1^0$. Moreover, $\mathcal F^{-1} \rho_i$ is a Schwartz function for fixed $i\ge -1$, and therefore $\mathbb R^d \ni z \mapsto \Delta_i \delta_z \in L^p$ is continuous. This easily yields that for $\epsilon>0$ the map $\R^d \ni z \mapsto \delta_z \in \cC^{-\epsilon}_1$ is continuous.
\begin{calc}
Indeed, we have $\Delta_i \delta_0 = \cF^{-1}(\rho_i \hat \delta_0) = \cF^{-1}(\rho_i) = \cF^{-1}(\rho(2^{-i} \cdot)) = 2^{id} \cF^{-1}(\rho)(2^i \cdot) $ and thus for $i\ge 0$
\begin{align*}
\|\Delta_i \delta_0 \|_{L^p} 
& = 2^{id} (\int \cF^{-1}(\rho)(2^i x)^p \dd x)^{\frac1p} 
= 2^{id} (2^{-\frac{di}{p}} \int \cF^{-1}(\rho)(x)^p \dd x)^{\frac1p} \\
& = 2^{i d (1- \frac{1}{p})} \| \cF^{-1}(\rho) \|_{L^p}. 
\end{align*}
Hence $\delta_0 \in B_{p,\infty}^{-d(1-\frac{1}{p})}$, similarly $\delta_z \in \cC_p^{-d(1-\frac{1}{p})}$. 
\end{calc}
\end{obs}

\begin{corollary}[of Proposition~\ref{proposition:solution_cauchy_problem_in_cC}]
\label{corollary:limit_gamma_n}
Let $\alpha \in (0,\frac12)$  and $b\in   C([0,T], B_{\infty,1}^{-\alpha}(\R^d,\R^d)) $. 
  
For $t\in [0,T)$ and $n\in \N$ let 
$ b^{(n)}_t= \sum_{i=1}^n \Delta_i b_t  \in C_{\rm b}^\infty(\R^d,\R^d)$
and let 
$\Gamma_{t,T}(x,y) = u^{\delta_y,b}(t,x)$ and $\Gamma_{t,T}^{(n)}(x,y)=u^{\delta_y,b^{(n)} }(t,x)$ (notation as in Proposition~\ref{proposition:solution_cauchy_problem_in_cC}).  
Then $\Gamma_{t,T}$ and $\Gamma_{t,T}^{(n)}$ are continuous on $\R^d \times \R^d$ and we have  for all $t\in [0,T)$ and $\mu \in \N_0^d$ with $|\mu|\le 1$: 
\begin{align*}
\sup_{x,y\in \R^d}
|\partial^\mu_x [ \Gamma_{t,T}(x,y) - \Gamma_{t,T}^{(n)}(x,y)]| \xrightarrow{n\rightarrow \infty} 0 . 
\end{align*}
\end{corollary}

\begin{proof}
The continuity follows from Proposition~\ref{proposition:solution_cauchy_problem_in_cC}. 
 
Because there exists a $C>0$ such that $\| b^{(n)}_s - b^{(n)}_r  \|_{B_{\infty,1}^{-\alpha}} \le C \| b_s - b_r \|_{B_{\infty,1}^{-\alpha}}$ for all $n\in\N$, $s,r\in [0,\infty)$ and $\|b^{(n)}_s - b_s\|_{B_{\infty,1}^{-\alpha }} \rightarrow 0$ for all $s\in [0,\infty)$ we obtain by a ``$3\epsilon$ argument'' that 
\begin{align*}
\|b^{(n)} - b\|_{C_t B_{\infty,1}^{-\alpha }}
 \rightarrow 0
\end{align*}
\begin{calc}
Indeed: Let $\epsilon>0$. Let $\delta>0$ be such that $|s-r|<\delta$ implies $\|b_s - b_r\|_{B_{\infty,1}^{-\alpha }} <\epsilon$. 
Let $S\subset [0,t]$ be a finite set such that for all $r\in [0,t]$ there is an $s\in S$ with $|r-s|<\delta$. Let $N\in \N$ be such that $\|b^{(n)}_s - b_s\|_{B_{\infty,1}^{-\alpha }} < \epsilon$ for all $n\ge N$ and $s\in S$. 
Then for all $r\in [0,t]$ we have with $s$ being such that $|s-r|<\delta$, for all $n\ge N$
\begin{align*}
\|b^{(n)}_r - b_r\|_{B_{\infty,1}^{-\alpha }}
\le 
\|b^{(n)}_r - b^{(n)}_s\|_{B_{\infty,1}^{-\alpha }}
+\|b^{(n)}_s - b_s\|_{B_{\infty,1}^{-\alpha }}
+\|b_s - b_r\|_{B_{\infty,1}^{-\alpha }} 
\le (M+2) \epsilon, 
\end{align*}
where $M>0$ is such that $\| b^{(n)}_s \|_{B_{\infty,1}^\alpha} \le M \| b_s \|_{B_{\infty,1}^\alpha}$ for all $n\in\N$ and $s\ge 0$. 
This follows by Youngs inequality as $\sum_{i=1}^n \Delta_i a= (\cF^{-1} \rho_{-1})(2^{-n} \cdot) * a$ for some $\rho_{-1}$ such that $\cF^{-1}$ is integrable, and because $\|\cF^{-1} \rho_{-1})(2^{-n} \cdot)\|_{L^1} = \|\cF^{-1} \rho_{-1})\|_{L^1}$ for all $n\in\N$. 
\end{calc}
 
As moreover $\sup_{y \in \R^d} \|\delta_y\|_{B_{1,\infty}^0} \lesssim 1$,
Proposition~\ref{proposition:solution_cauchy_problem_in_cC} yields 
\begin{align*}
\sup_{y \in \R^d} \| \Gamma_{t,T}(\cdot, y) - \Gamma_{t,T}^{(n)}(\cdot, y) \|_{\cC^\beta}\rightarrow 0,
\end{align*}
 for all $\beta < 2-\alpha$.
\end{proof}

\begin{proposition}
\label{proposition:transition_density}
Let $\alpha \in (0,\frac12)$ and $b\in   C([0,T],B_{\infty,1}^{-\alpha}(\R^d,\R^d)) $. For $  t\in [0,T)  $ let $\Gamma_{t,T} \colon \R^d \times \R^d \rightarrow \R$ be defined by $\Gamma_{t,T}(x,y) = u^{\delta_y}(  t  ,x)$. 
Let $\P_{t,x}$ be the unique probability measure on $C([t,T],\R^d)$ such that the coordinate process $X$ is a solution to the SDE \eqref{eqn:diffusion_with_drift_b_in_section} on $[t,T]$ with initial condition $X_t=x$. 
Then $\Gamma_{t,T}(x,\cdot)$ is the density of $X_T$ under $\P_{t,x}$, i.e., $\E_{t,x} [ \phi(X_T)] = \int_{\R^d} \phi(y) \Gamma_{t,T}(x,y) \dd y$ for all $\phi\in C_{\rm c}(\R^d)$. 
\end{proposition}
\begin{proof}
 
For $b$ with values in $C^\infty_{\rm b}$ this is classical, see for example  \cite[Theorem 6.5.4]{Fr75}. 
So let $b^{(n)}$ and $\Gamma_{t,T}^{(n)}$ be as in Corollary~\ref{corollary:limit_gamma_n} and for $x\in \R^d$ let $\P^{(n)}_{t,x}$
be the unique probability measure on $C([t,T],\R^d)$ such that the coordinate process $X$ is a solution to the martingale problem for $(  (\cL^{(n)}_s)_{s\in (t,T]}  ,\delta_x)$, where $\cL^{(n)}_s = \frac12 \Delta +   b_{T-s}^{(n)}  \cdot \nabla$. 
Using that $\P^{(n)}_{t,x}$ weakly converges to $\P_{t,x}$ (Theorem~\ref{theorem:martingale_solution}) and the uniform convergence in Corollary~\ref{corollary:limit_gamma_n} we obtain for $\phi \in C_{\rm c}(\R^d)$: 
\begin{align*}
\E_{t,x}[\phi(X_T)] = \lim_{n \to \infty} \E_{t,x}^{(n)}[\phi(X_T)] = \lim_{n \to \infty} \int_{\R^d} \phi(y) \Gamma_{t,T}^{(n)}(x,y) \dd y = \int_{\R^d} \phi(y) \Gamma_{t,T}(x,y) \dd y.
\end{align*}
 
\end{proof}

\section{Heat-kernel upper bounds}
\label{section:heat_kernel_upper_bound_for_smooth_drift}

Here we prove the upper bound~\eqref{eqn:heat_kernel_upper_bound} of the heat-kernel estimates. We follow the ``parametrix'' approach from Friedman's book~\cite{Fr64} to prove the heat-kernel estimates presented in Theorem~\ref{theorem:heat_kernel_bound}. This means that we write $\Gamma_t$ as a series (see Lemma~\ref{lemma:decomposition_of_Gamma}) and bound each term   in that series to obtain a bound for the whole series and thus for $\Gamma_t$.  Usually the point of the parametrix is to deal with non-constant diffusion coefficients, but the approach is
still useful for us despite the fact that we deal with constant diffusion coefficients.

Because of Corollary~\ref{corollary:limit_gamma_n} we can restrict our attention to $b$ in   $C([0,T],C_{\rm b}^\infty(\R^d, \R^d))$   and then extend the bounds to $b$ in   $C([0,T], B^{-\alpha}_{\infty,1}(\R^d, \R^d))$   by a limiting argument.

\noindent \textbf{For the rest of this section we fix $\alpha \in (0,\frac12)$,  and $c>1$ as in Theorem~\ref{theorem:heat_kernel_bound} and 
$b\in   C([0,\infty),C_{\rm b}^\infty(\R^d, \R^d))  
$.}   (Instead of $[0,T]$ we consider $[0,\infty)$ for notational convenience.)  
\begin{calc}
Observe that $C_{\rm b}^\infty \subset B_{\infty,1}^{-\alpha}$. 
\end{calc}

\begin{obs}
Let $g\in L^1(\R^d,\R^d)$ and $a\in C_{\rm b}^\infty(\R^d,\R^d)$. 
Let $(\tilde \rho_i)_{i \in \N_{-1}}$ be another dyadic partition of unity, but such that $\supp \tilde \rho_{-1} \cap \supp \rho_i = \emptyset$ for $i\in \N_0$ so that  
  
\begin{align*}
\int_{\R^d} ( \Delta_{i} a)(z) (\tilde \Delta_{-1} g)(z) \dd z 
& 
= \int_{\R^d} \cF^{-1}( \rho_i \hat a)(z) \cF^{-1}( \tilde \rho_{-1}   \hat g)(z) \dd z \\
& = \int_{\R^d} \hat a(-z)  \rho_i(z) \tilde \rho_{-1}(z) \hat g(z) \dd z =0,
\end{align*}
and thus 
\begin{align*}
\int_{\R^d} ( \Delta_{\ge 0} a)(z) g(z) \dd z 
= \int_{\R^d} ( \Delta_{\ge 0} a)(z) (\tilde \Delta_{\ge 0} g)(z) \dd z. 
\end{align*}
 
By duality and Bernstein's inequality, see \cite[Proposition 2.76 and Lemma 2.1]{BaChDa11}, we have 
\begin{align} \notag
   \Big| \int_{\mathbb{R}^d} a (z) \cdot &  g (z)   \dd z \Big| 
   \le      \left| \int_{\mathbb{R}^d} \Delta_{-1} a (z)  \cdot  g (z)   \dd z \right| +  \left| \int_{\mathbb{R}^d} \Delta_{\ge 0} a (z) \cdot  g (z)   \dd z \right| \\
  \notag & \lesssim 
   \| \Delta_{-1} a \|_{L^\infty} \|g\|_{L^1} 
+  \| \Delta_{\ge 0} a \|_{B_{\infty,1}^{-\alpha}} 
\| \tilde \Delta_{\ge 0} g \|_{B_{1,\infty}^{\alpha}} \\
\notag & \lesssim 
   \| \Delta_{-1} a \|_{L^\infty} \|g\|_{L^1} 
+  \| \Delta_{\ge0} a \|_{B^{-\alpha}_{\infty, 1}} \left( \sup_{j \ge 0 } \left\{ \|\tilde \Delta_j g\|_{L^1}^{1-\alpha} (2^{j} \|\tilde \Delta_j g\|_{L^1})^{\alpha} \right\} \right)
 \\
\label{eqn:estimate_b_g_integral_with_B_infty_1_norm} 
  & \lesssim 
   \| \Delta_{-1} a \|_{L^\infty} \|g\|_{L^1} 
+  \| \Delta_{\ge 0} a \|_{B_{\infty,1}^{-\alpha}} \| g \|_{L^1}^{1 - \alpha} \| \nabla g\|_{L^1}^\alpha . 
\end{align}
\end{obs}

We will apply the above bound for functions $g$ that are Gaussian, therefore we will need estimates for derivatives of Gaussian functions. So we recall the following bound:
\begin{obs}
Let $p(t,x) = (2\pi t)^{-\frac{d}{2}} e^{- \frac{1}{2t} |x|^2  }$ for $(t,x) \in (0,\infty) \times \R^d$ be the standard Gaussian kernel. 
For the space derivatives $\partial^\mu p$ we have the following estimate:
\begin{align} 
\label{eqn:bound_partial_mu_p}
\forall \mu \in \N_0^d 
\ \exists C>0 
\ \forall (t,x) \in (0,\infty) \times \R^d : \quad 
|\partial^{\mu} p (t, x)| \le C t^{- \frac{| \mu |}{2}} p (c t, x),
\end{align}
\end{obs}

\begin{calc}
Indeed, when $\mu$ has a $0$ for the $t$ derivatives, then $\partial^\mu = t^{-\frac{|\mu|}{2}} P_\mu(\frac{x}{\sqrt{t}}) p(t,x)$ for some polynomial $P_\mu$ of order $\le |\mu|$. This can be proved by induction, for the induction step note that $\partial_{x_i} P_\mu(\frac{x}{\sqrt{t}})$ equals the product of another polynomial with  $t^{-\frac12}$. 
Note that one can bound $P(a) e^{-a}$ by $C e^{-ca}$ for all $c<1$, where $C$ depends on $c$. 
\end{calc}

The proof of the upper bound~\eqref{eqn:heat_kernel_upper_bound} essentially follows by iterating the previous two observations. To carry out the argument we need the following result, which allows us to write $\Gamma$ as an infinite series.


\begin{lemma}
\label{lemma:decomposition_of_Gamma}
Let $t>0$ and $y\in\R^d$. 
For $s\in [0,t)$ and $x \in \R^d$ we define 
\begin{align} 
\label{eqn:Psi_1}
\Psi^{y,1}_{s,t} ( x ) & =  - b (  t-s,   x) \cdot \nabla p (  s  , x - y). 
\end{align}
Then for all $k\in \N$ the map $s\mapsto \Psi^{y,k}_{s,t}$ is in $L^1([0,t), L^1(\R^d))$, where 
\begin{align}
\label{eqn:definition_Psi_k+1}
 \Psi^{y,k+1}_{s,t} ( x ) & = - \int_0^{s} \int_{\mathbb{R}^d} b (  t-s,   x) \cdot 
   \nabla p (  s-r  , x - z) \Psi^{y,k}_{r,t} ( z ) \dd z \dd r.
\end{align}
Moreover, (with $\Gamma_{s ,t}$ as in Proposition~\ref{proposition:transition_density})
\begin{align} 
\label{eqn:Gamma_in_series}
\Gamma_{s,t} (x , y) = 
p (t-s, x - y) + \sum_{k = 1}^{\infty} \int_0^{t-s}  \int_{\mathbb{R}^d} p (   t-s-r   ,  x - z) \Psi^{y,k}_{r,t}( z ) \dd z \dd r.
\end{align}
\end{lemma}
\begin{proof}
By~\eqref{eqn:bound_partial_mu_p} we know that $\|\Psi^{y,1}_{s,t} \|_{L^1(\R^d)} \lesssim \| \nabla p(s,\cdot)\|_{L^1(\R^d)} \lesssim  s^{-\frac{1}{2}}  $ and therefore $s\mapsto \Psi^{y,1}_{s,t}$ is in $L^1([0,t), L^1(\R^d))$. 
Observe that $\Psi_{s,t}^{y,k+1}$ equals the inner product of $-b(t-s,x)$ with a convolution in space and time. Therefore, by applying the $L^1$ inequality for convolutions (Young's inequality) for the space as well for the time convolution, we obtain 
\begin{align*}
\|\Psi^{y,k+1} _{s,t}\|_{L^1(\R^d)}  
& \lesssim \int_0^{s} \| \nabla p (s-r, \cdot)\ast \Psi^{y,k}_{r,t}\|_{L^1(\R^d)} \dd r \\
   & \lesssim \int_0^{s} (s-r)^{-\frac{1}{2}} \dd r \int_0^s \| \Psi^{y,k}_{r,t}\|_{L^1(\R^d)}  \dd r 
   \lesssim s^{\frac12} \int_0^t \| \Psi^{y,k}_{r,t}\|_{L^1(\R^d)}  \dd r,
\end{align*}
from which we conclude that $\int_0^t \| \Psi^{y,k}_{r,t}\|_{L^1(\R^d)}  \dd r$ is finite  (actually it is $\lesssim t^{\frac{1+3k}{2}}$) for all $k\in\N$.

It remains to show~\eqref{eqn:Gamma_in_series}. 
As $\Gamma_{s,t}(x,y) = u^{\delta_y}(s,x)$ where $u^{\delta_y}$ being the fixed point of the map $\Phi$ as in \eqref{eqn:map_for_cauchy_fixed} with $\phi = \delta_y$, that is, with $u = u^{\delta_y}$,
\begin{align*}
(\Phi u)_s 
&  =  P_{t-s} \delta_y  -  \int_s^t P_{q-s} (  b_{q}   
 \cdot \nabla u_q) \dd q \\
& =  P_{t-s} \delta_y  -  \int_0^{t-s} P_{t-s-r} (  b_{t-r}   
 \cdot \nabla u_{t-r}) \dd r.
\end{align*}
From a Picard iteration it follows that $\Gamma$ is the limit of the sequence $\Gamma_t^0 = 0$,
\begin{align*} 
& \Gamma^{k+1}_{s,t}(x,y) \\
& = p(t-s,x-y) - \int_0^{t-s} \int_{\R^d} p(t-s-r,x-z) (b(t-r, z) \cdot \nabla_z \Gamma^k_{t-r,t}(z,y)) \dd z \dd r . 
\end{align*}
Therefore,   $\Gamma^1_{s,t}(x,y) = p(t-s,x-y)$   and we obtain recursively 
\begin{calc}
\begin{align*}
\Gamma_{s,t}^2(x,y) 
& = p(t-s,x-y) - \int_0^{t-s} \int_{\R^d} p(t-s-r, x-z) b(t-r,z) \cdot \nabla p(r, z-y) \dd z \dd r \\
& = p(t-s,x-y) + \int_0^{t-s} \int_{\R^d} p(t-s-r, x-z) \Psi^{y,1}_{r,t}(z) \dd z \dd r \\
\Gamma_{s,t}^3(x,y) 
& = p(t-s,x-y) + \int_0^{t-s} \int_{\R^d} p(t-s-r, x-z) \Psi^{y,1}_{r,t}(z) \dd z \dd r \\
& \quad 
- \int_0^{t-s} \int_{\R^d} p(t-s-r, x-z)   b(t-r,z) \cdot \nabla 
\left( \int_0^{r} \int_{\R^d} p(r-q, z-w) \Psi^{y,1}_{q,t}(w) \dd w \dd q \right) \dd z \dd r \\
& = p(t-s,x-y) + \int_0^{t-s} \int_{\R^d} p(t-s-r, x-z) \Psi^{y,1}_{r,t}(z) \dd z \dd r \\
& \quad 
+ \int_0^{t-s} \int_{\R^d} p(t-s-r, x-z)   \Psi^{y,2}_{r,t}(z)  \dd z \dd r
\end{align*}
\end{calc}
(see also \cite[Chapter~1.4]{Fr64})
 
\begin{align*}
 \Gamma^{k+1}_{s,t}(x,y) 
& = p(t-s,x-y) +   \sum_{\ell=1}^k \int_0^{t-s} \int_{\R^d} p(t-s-r,x-z) \Psi^{y,\ell}_{r,t}(z) \dd z \dd r. 
\end{align*}
 
This proves~\eqref{eqn:Gamma_in_series}.
\end{proof}

\begin{obs}
\label{obs:restricting_to_one_parameter}
Now let us get back to Remark~\ref{remark:inhomogeneous}. 
Observe that in the right-hand side in \eqref{eqn:Gamma_in_series} the dependence on $t$ is in the $\Psi^{y,k}$ functions, and we see that the rest is a function of $t-s$. 
This allows us to take the first time variable, $s$, equal to zero, and proof the heat-kernel bounds as in Theorem~\ref{theorem:heat_kernel_bound}.  From now on we write ``$\Gamma_t$'' for ``$\Gamma_{0,t}$''. 
\end{obs}

Note that the first term appearing in the right-hand side of \eqref{eqn:Gamma_in_series} is already bounded by the right-hand side of \eqref{eqn:heat_kernel_upper_bound}. 
Therefore, we will recursively estimate
\begin{align*}
\int_0^t \int_{\mathbb{R}^d} p (t - s, x - z) \Psi^{y,k}_{s,t}(z) \dd z \dd s. 
\end{align*}
This will be done with the help of some auxiliary lemmas, which follow below. 

\begin{obs}
\label{obs:notations_P_mu_and_cP}
Let $\mu \in \N_0^d$, $t>0$, $k\in \N$, $x,y\in \R^d$ and $g\in L^1(\R^d)$. 
As we write $P_t g = p(t,\cdot) \ast g$ (see \eqref{eqn:map_for_cauchy_fixed}), we have $\partial^\mu P_t g= \partial^{\mu} p (t, \cdot) \ast g$. 

For any given norm $\|\cdot\|$ we will write $\|\nabla f\| = \sum_{i=1}^d \|\partial_i f\|$ and $\|\nabla^2 f\| = \sum_{i,j=1}^d \|\partial_{i}\partial_j f\|$. 
\end{obs}

\begin{lemma}
\label{lemma:bound_for_partial_mu_P_g_r}
There exists a $C > 0$ (independent of $b$) such that for all $\mu\in \N_0^d$ with $| \mu | \le 2$,  $y \in \mathbb{R}^d$ and $t,s,r\in (0,\infty)$ with $t>s>r$ and all $f\in L^1(\R^d)$, with   $g_{t,s,r}(z) = b(  t-s,   z) \cdot \int_{\R^d}  \nabla p(s-r,  z-w) f(w) \dd w$ 
   \begin{align}
  \notag 
 &    | \partial^\mu P_{t-s} g_{t,s,r} (x) | 
     \le C  (t - s)^{- \frac{|\mu|}{2}} p (c t, x - y) \Big( \| \Delta_{-1}   b_{t-s}   \|_{L^\infty}   \left\| \tfrac{ \nabla P_{s-r}f }{p (c s, \cdot -     y)} \right\|_{L^{\infty}}  \\
 \label{eqn:bound_cP_mu_g}
& + 
\| \Delta_{\ge 0}   b_{t-s}    \|_{B^{-  \alpha}_{\infty, 1}} 
 \Big[ (t -    s)^{- \frac{\alpha}{2}}  \left\| \tfrac{ \nabla P_{s-r}f }{p (c s, \cdot -     y)} \right\|_{L^{\infty}} + 
  \left\| \tfrac{ \nabla P_{s-r}f }{p (c s, \cdot -     y)} \right\|_{L^{\infty}}^{1 - \alpha} 
 \left\| \tfrac{ \nabla^2 P_{s-r}f }{p (c s, \cdot - y)} \right\|_{L^{\infty}}^{\alpha} \Big] \Big) . 
  \end{align}
\end{lemma}

\begin{proof}
We abbreviate $g_{t,s,r}$ by $g$. 
Observe that $g(z)=  b(  t-s,   z) \cdot \nabla  P_{s-r} f(z)$. 
Then, with $h: \R^d \rightarrow \R^d$,  $h (z)= \partial^\mu p(t-s, x-z) \nabla P_{s-r} f(z)$,  
by \eqref{eqn:estimate_b_g_integral_with_B_infty_1_norm} 
    \begin{align*}
     | \partial^\mu P_{t-s} g (x) | & = \left| \int_{\R^d} \partial^\mu p(t-s,x-z)  b(  t-s,   z) \cdot \nabla P_{s-r} f(z) \dd z \right| \\
     & \lesssim \| \Delta_{-1}   b_{t-s}   \|_{L^\infty} \|h\|_{L^1}
     +  \| \Delta_{\ge 0}   b_{t-s}   \|_{B^{-\alpha}_{\infty, 1}} 
 \|h\|_{L^1}^{1 - \alpha} \| \nabla h \|_{L^1}^{\alpha} .
  \end{align*}

We estimate both $\|h\|_{L^1}$ and $\|\nabla h\|_{L^1}$. 
We use \eqref{eqn:bound_partial_mu_p} and  $\int_{\R^d} p( c(t-s), x-z) p(cs, z-y) \dd z = p( c (t-s) , \cdot) * p ( cs, \cdot) (x-y)  = p( ct , x-y)$ to obtain   
  \begin{align*}
 \|h\|_{L^1}
    & =    \int_{\mathbb{R}^d} \left| \partial^{\mu} p (t - s, x - z) \nabla P_{s-r}f(z) 
    \right| \dd z\\
&      \lesssim \int_{\mathbb{R}^d} (t - s)^{- \frac{|\mu|}{2}} p (c (t - s), x - z) 
    p (c s, z - y)  \left\| \frac{\nabla P_{s-r}f}{p (c s, \cdot -
    y)} \right\|_{L^{\infty}} \dd z\\
&      = (t - s)^{- \frac{|\mu|}{2}} p (c t, x - y) \left\|  \frac{\nabla P_{s-r}f}{p (c s, \cdot - y)} \right\|_{L^{\infty}} .
  \end{align*}
Similarly, in combination with Leibniz's rule, we obtain
  \begin{align*}
& \| \nabla h\|_{L^1} = \big\| \nabla \big(
    \partial^{\mu} p (t - s, x - \cdot) \nabla P_{s-r}f \big) \big\|_{L^1} \\
& 
    \le \sum_{i=1}^d \big\|  \partial^{ \mu} \partial_i  p (t - s, x - \cdot) \nabla P_{s-r}f
     \big\|_{L^1}     
    +   \big\| \partial^{\mu} p (t - s, x - \cdot)
   \nabla^2 P_{s-r} f \big\|_{L^1} \\
&     \lesssim (t - s)^{- \frac{|\mu|}{2}} p (c t, x - y)
 \left[ (t - s)^{- \frac{1}{2}} 
    \left\| \frac{ \nabla P_{s-r}f }{p (c s, \cdot - y)} \right\|_{L^{\infty}} 
      +  \left\|  \frac{\nabla^2 P_{s-r}f }{p (c s, \cdot - y)} \right\|_{L^{\infty}} \right].
  \end{align*}
Using the above and that $(a+b)^\alpha \le a^\alpha + b^\alpha$ for $a,b\ge 0$ 
\begin{calc}
(indeed, by dividing by $b^\alpha$ we may instead assume $b=1$; then, by a simple calculation we see that $(x+1)^\alpha \le x^\alpha +1$: $\partial_x x^\alpha +1 - (x+1)^\alpha = \alpha ( x^{\alpha-1} - (x+1)^{\alpha -1}) \ge 0$ for $x\ge 0$)
\end{calc}
we obtain \eqref{eqn:bound_cP_mu_g}. 
\begin{calc}
\begin{align*}
\|h\|_{L^1}^{1-\alpha} \| \nabla h\|_{L^1}^\alpha
& \lesssim 
(t - s)^{- \frac{|\mu|}{2}} p (c t, x - y) \left\| \frac{ \nabla P_{s-r}f }{p (c s, \cdot - y)} \right\|_{L^{\infty}}^{1-\alpha} \\
& \qquad \times 
 \left[ (t - s)^{- \frac{1}{2}}  \left\| \frac{ \nabla P_{s-r}f }{p (c s, \cdot - y)} \right\|_{L^{\infty}} 
      +  \left\|  \frac{ \nabla^2 P_{s-r}f }{p (c s, \cdot - y)} \right\|_{L^{\infty}} \right]^\alpha \\
& \lesssim 
(t - s)^{- \frac{|\mu|}{2}} p (c t, x - y) \left\| \frac{ \nabla P_{s-r}f }{p (c s, \cdot - y)} \right\|_{L^{\infty}}^{1-\alpha} \\
& \qquad \times 
 \left[ (t - s)^{- \frac{\alpha}{2}}  \left\| \frac{ \nabla P_{s-r}f }{p (c s, \cdot - y)} \right\|_{L^{\infty}}^ \alpha
      +  \left\|  \frac{ \nabla^2 P_{s-r}f }{p (c s, \cdot - y)} \right\|_{L^{\infty}}^\alpha \right].
\end{align*}
\end{calc}
\end{proof}

\begin{obs}
Now we apply the above lemma to our setting. But first, let us introduce some notation. For $k \in \N$,  $t \ge 0$, $i \in \{0,1\}$, and $\beta \in \{0,\alpha\}$ we write 
\begin{align*}
\cI_{i,k}^\beta(t) = \sup_{y \in \R^d} \int_0^t 
 \left\| \frac{ \nabla^i P_{t-s}[ \Psi^{y,k}_{s,t}]}{p(ct,\cdot -y)} \right\|_{L^\infty}^{1-\beta}  \left\| \frac{ \nabla^{i+1} P_{t-s}[ \Psi^{y,k}_{s,t}]}{p(ct,\cdot -y)} \right\|_{L^\infty}^{\beta} \dd s.
\end{align*}
We are interested in the bounds for $\cI_{i,k}^0$ only. 
But in order to describe a recursive relation for them, as we will see in the next lemma, we also need the $\cI_{i,k}^\alpha$'s.  
\end{obs}

\begin{lemma}
\label{lemma:recursive_relation_cI}
Let $C>0$ be as in Lemma~\ref{lemma:bound_for_partial_mu_P_g_r}. 
For all $k\in \N$, $t\ge 0$, $i\in \{0,1\}$ and $\beta \in \{0,\alpha\}$  
\begin{align}
\notag 
 \cI_{i,k+1}^\beta (t) \le
C \int_0^t (t-s)^{-\frac{i+\beta}{2}} \Big( 
& \| \Delta_{-1} b\|_{  C_t   L^\infty} \cI_{1,k}^0(s) \\
\label{eqn:recursive_relations_cI}
&   + \| \Delta_{\ge 0} b\|_{  C_t   B_{\infty,1}^{-\alpha} } [ (t-s)^{-\frac{\alpha}{2}} \cI_{1,k}^0(s) + \cI_{1,k}^\alpha(s) ]\Big) \dd s . 
\end{align}
\end{lemma}
\begin{proof}
We claim that the following holds. 
For all $k \in \N$, $y\in \R^d$ and $i \in \{0,1,2\}$ 
\begin{align}
\notag &  \left\| \tfrac{ \nabla^i P_{t-s}[\Psi^{y,k+1}_{s,t}] }{p (c t, \cdot -     y)} \right\|_{L^{\infty}} 
     \le C  (t - s)^{- \frac{i}{2}}  \bigg( \| \Delta_{-1} b\|_{  C_t   L^\infty} \int_0^s 
 \left\| \tfrac{ \nabla P_{s-r}[\Psi^{y,k}_{r,t}] }{p (c s, \cdot -     y)} \right\|_{L^{\infty}} \dd r \\
\notag &  \qquad    + \| \Delta_{\ge 0} b\|_{  C_t   B_{\infty,1}^{-\alpha} } \Big[
 (t -    s)^{- \frac{\alpha}{2}}
\int_0^s 
 \left\| \tfrac{ \nabla P_{s-r}[\Psi^{y,k}_{r,t}] }{p (c s, \cdot -     y)} \right\|_{L^{\infty}} \dd r  \\
& \hspace{3.5cm} + 
  \int_0^s \left\| \tfrac{\nabla P_{s-r}[\Psi^{y,k}_{r,t}] }{p (c s, \cdot -     y)} \right\|_{L^{\infty}}^{1 - \alpha} 
  \left\| \tfrac{\nabla^2 P_{s-r}[\Psi^{y,k}_{r,t}] }{p (c s, \cdot - y)} \right\|_{L^{\infty}}^{\alpha} \dd r \Big] \bigg).
  \label{eqn:bound_P_Psi_etc}
\end{align}
From this \eqref{eqn:recursive_relations_cI} follows by definition of $\cI_k^\beta$. 
Now let us prove \eqref{eqn:bound_P_Psi_etc}. 
Let $g_{t,s,r}$ be as in Lemma~\ref{lemma:bound_for_partial_mu_P_g_r} with $f = \Psi^{y,k}_{r,t}$. 
Observe that by definition of $\Psi^{y,k+1}_{s,t}$ \eqref{eqn:definition_Psi_k+1} 
we can write 
\begin{align*}
\Psi^{y,k+1}_{s,t}(z)
= \int_0^s b(  t-s ,  z) \cdot \nabla  P_{s-r}[ \Psi^{y,k}_{r,t}](z) \dd r 
 = \int_0^s  g_{t,s,r}(z) \dd r ,
\end{align*}
so that (one can verify the interchange of integrals by Fubini's theorem and 
using Lemma~\ref{lemma:decomposition_of_Gamma})
  \begin{align*}
   | \nabla^i P_{t-s}[\Psi^{y,k+1}_{s,t}] (x) |
&     \le  \int_0^s       | \nabla^i P_{t-s} g_{t,s,r} (x) | \dd r .
  \end{align*}
With this, \eqref{eqn:bound_P_Psi_etc} follows from \eqref{eqn:bound_cP_mu_g}. 
\end{proof}

In the proof of Lemma~\ref{lemma:bound_cIs} we will use the following bound for the beta function (see \eqref{eqn:beta_function_def}).

\begin{lemma}
\label{lemma:bound_on_beta_function}
Let $\delta \in ( 0,1] $. 
Then $M_\delta := \sup \{ B(\beta, \gamma) \gamma^\beta : (\beta ,\gamma) \in [\delta, 1] \times [\delta,\infty) \} <\infty$. 
Hence, for all $(\beta, \gamma) \in [\delta, 1] \times [\delta,\infty)$, 
\begin{align*}  
B(\beta, \gamma) = B(\gamma, \beta) \le M_{\delta} \gamma^{- \beta}. 
\end{align*}
\end{lemma}
\begin{proof}
By 
\cite[Theorem 1.1.4 and Theorem 1.4.1]{AnAsRo99} we have for $\gamma,\beta >0$
\begin{align*}
B(\beta, \gamma) = \frac{\Gamma(\gamma) \Gamma(\beta)}{\Gamma(\gamma+\beta)}, 
\quad \mbox{ and } \quad 
\lim_{\gamma \rightarrow \infty} \frac{ \Gamma (\gamma)}{ \sqrt{2\pi} \gamma^{\gamma -\frac12} e^{-\gamma}} =1 . 
\end{align*}
From this we deduce the following. 
Let $\beta_n \rightarrow \beta$ for some $\beta \in [\delta,1]$ and $\gamma_n \rightarrow \infty$. 
Then 
\begin{align*}
\lim_{n \rightarrow \infty} \frac{B(\beta_n, \gamma_n) \gamma_n^{\beta_n}}{\Gamma(\beta_n)} 
& = 
\lim_{n \rightarrow \infty} \frac{ \sqrt{2\pi} \gamma_n^{\gamma_n -\frac12} e^{-\gamma_n} \gamma_n^{\beta_n}}{ \sqrt{2\pi} (\gamma_n+\beta_n)^{\gamma_n + \beta_n -\frac12} e^{-(\gamma_n +\beta_n) } } 
\\
& = 
\lim_{n \rightarrow \infty}  (1+\frac{\beta_n}{\gamma_n})^{- (\gamma_n + \beta_n -\frac12)} e^{ \beta_n } \\
& = 
\lim_{\gamma \rightarrow \infty}  (1+\frac{\beta_n}{\gamma_n})^{- \gamma_n } e^{ \beta_n } = e^{-\beta_n} e^{\beta_n} = 1. 
\end{align*}
Therefore 
\begin{align*}
\lim_{n \rightarrow \infty} B(\beta_n, \gamma_n) \gamma_n^{\beta_n} = \Gamma(\beta), 
\end{align*}
so that from the continuity of $\Gamma$ it follows that $(\beta, \gamma) \mapsto B(\beta,\gamma) \gamma^\beta$ is a bounded function on $[\delta,1] \times [\delta,\infty)$. 
\end{proof}

Let us now use the recursive relation for $\cI_{i,k}^\beta$ and the bounds on the beta function to obtain estimates for $\cI_{i,k}^\beta$: 

\begin{lemma}
\label{lemma:bound_cIs}
Let $C>0$ be as in Lemma~\ref{lemma:bound_for_partial_mu_P_g_r}  
and let $M = 8 M_{\frac{1}{2} - \alpha}$ with $M_\delta$ as in Lemma~\ref{lemma:bound_on_beta_function}. 
There exists a $K > 0$   (independent of $b$)    such that 
for all $k \in \N$, $t>0$, $\beta \in \{0,\alpha\}$ and $i \in \{0,1\}$ 
\begin{align}
\label{eqn:bound_cI}
\cI_{i,k}^\beta(t) \le 
K \sum_{\substack{ m,n\in \N_0:\\ m+n=k} } t^{-\frac{i+\beta}{2}}
\frac{(CM \| \Delta_{-1} b\|_{  C_t   L^\infty} t^{\frac12} )^m }{( m!)^{\frac{1-\beta}{2} } }
\frac{(CM \| \Delta_{\ge 0} b\|_{  C_t   B_{\infty,1}^{-\alpha} }  t^{\frac{1-\alpha}{2} } )^n }{(n!)^{\frac{1-\alpha-\beta}{2} } }. 
\end{align}
\end{lemma}

\begin{proof}
We give a proof by induction. Instead of ``$\| \Delta_{-1} b\|_{  C_t   L^\infty}$'' and ``$\| \Delta_{\ge 0} b\|_{  C_t   B_{\infty,1}^{-\alpha} }$'' we will write ``$X$'' and ``$Y$'', respectively. 

\noindent $\bullet$ The induction start, $k=1$: \\
We have for $\mu \in \N_0^d$ with $|\mu|\le 2$ 
\begin{align*}
\partial^\mu P_{t-s}[ \Psi^{y,1}_{s,t}](x) = \int_{\R^d} \partial^\mu p(t-s,x-z) \Psi^{y,1}_{s,t}(z) \dd z = \int_{\R^d} b(z) \cdot g_\mu(z) \dd z
\end{align*} 
with $g_\mu(z)= \nabla p(s,z-y) \partial^\mu p(t-s,x-z)$. 
By \eqref{eqn:bound_partial_mu_p} there exists a $  K  >0$ such that for all $\mu, \nu \in \N_0^d$ with $|\mu|\le 2$ and $|\nu| \le 1$:
\begin{align*}
|g_\mu(z)| &\le   K   (t-s)^{-\frac{|\mu|}{2}} s^{-\frac12} p(cs,z-y) p(c(t-s), x-z), \\
|\partial^\nu g_\mu(z)| &\le   K   (t-s)^{-\frac{|\mu|}{2}} s^{-\frac12} [(t-s)^{-\frac12} + s^{-\frac12}] p(cs,z-y) p(c(t-s), x-z).
\end{align*}
Therefore, by \eqref{eqn:estimate_b_g_integral_with_B_infty_1_norm}, for $j \in \{0,1,2\}$
\begin{calc}
\begin{align*}
\left|  \nabla^j P_{t-s}[ \Psi^{y,k}_{s,t}] (x)\right|
\le   K   (t-s)^{-\frac{j}{2}} s^{-\frac{1}{2}} \Big( X + Y [(t-s)^{-\frac{\alpha}{2}} + s^{-\frac{\alpha}{2}}]\Big) p(ct,x-y),
\end{align*}
and thus 
\end{calc}
\begin{align*}
\left\| \frac{ \nabla^j P_{t-s}[ \Psi^{y,1}_{s,t}]}{p(ct,\cdot -y)} \right\|_{L^\infty} 
\le   K   (t-s)^{-\frac{j}{2}} s^{-\frac{1}{2}} \Big( X + Y [(t-s)^{-\frac{\alpha}{2}} + s^{-\frac{\alpha}{2}}]\Big),
\end{align*}
so that for $i \in \{0,1\}$ 
\begin{align*}
\cI_{i,1}^\beta(t)
& \le   K   \int_0^t 
(t-s)^{-\frac{i+\beta}{2}} s^{-\frac{1}{2}} \Big( X + Y [(t-s)^{-\frac{\alpha}{2}} + s^{-\frac{\alpha}{2}}]\Big) \dd s \\
& \le  t^{-\frac{i+\beta}{2}}   K  
 \Big(  B(\tfrac{2-i-\beta}{2}, \tfrac12) X   t^{\frac12}   + \big[B(\tfrac{2-i-\alpha-\beta}{2},\tfrac12) + B(\tfrac{2-i-\beta}{2}, \tfrac{1-\alpha}{2}) \big] Y   t^{\frac{1-\alpha}{2}}    \Big) . 
\end{align*}
Hence, for $k=1$, the inequality \eqref{eqn:bound_cI} follows by applying Lemma~\ref{lemma:bound_on_beta_function} for the beta functions and using that $\delta \mapsto M_\delta$ is decreasing: 
\begin{align*}
B(\tfrac{2-i-\beta}{2}, \tfrac12) 
& \le M_{\frac{2-i-\beta}{2}} (\frac12)^{-\frac{2-i-\beta}{2}} \le 2 M_{\frac{1}{2}-\alpha} \le M, \\
B(\tfrac{2-i-\alpha-\beta}{2},\tfrac12) 
& \le M_{\frac{2-i-\alpha-\beta}{2}} 2^{\frac{1-\alpha-\beta}{2}} \le M, \\ 
B(\tfrac{2-i-\beta}{2}, \tfrac{1-\alpha}{2}) 
& \le M_{\frac{2-i-\beta}{2}} (\tfrac{1-\alpha}{2})^{-\frac{1-\beta}{2}} 
\le M_{\frac12-\alpha} 4^{\frac{1-\beta}{2}} \le M. 
\end{align*}

\noindent $\bullet$ The induction step, from $k$ to $k + 1$: \\
Let $k\in \N$ and assume that \eqref{eqn:bound_cI} holds. Then by Lemma~\ref{lemma:recursive_relation_cI} 
\begin{align*}
\cI_{i,k+1}^\beta (t)
& \le C \int_0^t (t-s)^{-\frac{i+\beta}{2}} \left( X \cI_{1,k}^0(s) + Y [ (t-s)^{-\frac{\alpha}{2}} \cI_{1,k}^0(s) + \cI_{1,k}^\alpha(s) ] \right)\dd s \\
& \le   K   C \sum_{\substack{ m,n\in \N_0:\\ m+n=k} } 
\frac{(CM X  )^m }{( m!)^{\frac{1-\beta}{2}} }
\frac{(CM Y  )^n }{(n!)^{\frac{1-\alpha-\beta}{2} } } \\
& \qquad \times 
 \int_0^t (t-s)^{-\frac{i+\beta}{2}} s^{-\frac{1}{2} +\frac{m}{2}+ n \frac{1-\alpha}{2}} \left( X  + Y [ (t-s)^{-\frac{\alpha}{2}}   + s^{-\frac{\alpha}{2}}] \right) \dd s . 
\end{align*}
We bound the latter integral, for which we have the following identity:
\begin{align*}
\MoveEqLeft  \int_0^t (t-s)^{-\frac{i+\beta}{2}} s^{-\frac{1}{2} +\frac{m}{2}+ n \frac{1-\alpha}{2}} \left( X  + Y [ (t-s)^{-\frac{\alpha}{2}}   + s^{-\frac{\alpha}{2}} ] \right)\dd s\\ 
& = t^{-\frac{i+\beta}{2}} t^{\frac{m}{2}+ n \frac{1-\alpha}{2}}
\Big( X   t^{\frac12}   B(\tfrac{1-\beta}{2}, \tfrac{m+1 + n(1-\alpha)}{2}) \\
& \qquad + Y   t^{\frac{1-\alpha}{2}}   \big[ B(\tfrac{1-\alpha-\beta}{2},  \tfrac{m+1 + n(1-\alpha)}{2})  
 + B(\tfrac{1-\beta}{2},  \tfrac{m + (n+1)(1-\alpha)}{2})  \big] \Big) .
\end{align*}
This shows that the power of $t$ is the right one. 
We bound the beta function terms to finish the proof. 
By Lemma~\ref{lemma:bound_on_beta_function} we have 
\begin{align*}
B(\tfrac{1-\beta}{2}, \tfrac{m+1 + n(1-\alpha)}{2}) 
& \le M_{\frac{1-\beta}{2}} \left( \tfrac{m+1 + n(1-\alpha)}{2} \right)^{-\frac{1-\beta}{2}} 
 \le 4 M_{\frac{1}{2}-\alpha}  \left(  m+1 \right)^{-\frac{1-\beta}{2}} ,\\
B(\tfrac{1-\alpha-\beta}{2}, \tfrac{m+1 + n(1-\alpha)}{2}) 
& \le M_{\frac{1-\alpha - \beta}{2}} \left( \tfrac{m+1 + n(1-\alpha)}{2} \right)^{-\frac{1-\alpha-\beta}{2}} 
 \le 4 M_{\frac{1}{2}-\alpha}  \left(  n+1 \right)^{-\frac{1-\alpha-\beta}{2}}, \\
 B(\tfrac{1-\beta}{2}, \tfrac{m + (n+1)(1-\alpha)}{2}) 
& \le M_{\frac{1-\beta}{2}} \left( \tfrac{m + (n+1)(1-\alpha)}{2} \right)^{-\frac{1-\beta}{2}} 
 \le 4 M_{\frac{1}{2}-\alpha}  \left(  n+1 \right)^{-\frac{1-\alpha-\beta}{2}} . 
\end{align*}
\begin{calc}
Here we used that $\frac{1-\alpha}{2} \in (\frac14, \frac12)$
\begin{align*}
\left(\frac{1-\alpha}{2}\right)^{- \frac{1-\alpha-\beta}{2}}
\le 4^{ \frac{1-\alpha-\beta}{2}} \le 4. 
\end{align*}
\end{calc}
\end{proof}

\begin{remark}
The restriction $\alpha \in (0,\frac12)$ in Lemma~\ref{lemma:bound_cIs} is necessary since $M = 4 M_{\frac12 - \alpha}$ diverges as $\alpha \uparrow \frac12$ (see see the definition of $M_\delta$ in Lemma~\ref{lemma:bound_on_beta_function}). 
This is not unexpected, since for $\alpha > \frac12$ we are no longer in the Young regime and we would need techniques like paracontrolled distributions or regularity structures to solve the equation for $\Gamma$. 
\end{remark}

Lemma~\ref{lemma:bound_cIs} together with the following basic inequality constitutes the proof of Theorem~\ref{theorem:heat_kernel_bound}. 

\begin{lemma}
\label{lemma:bound_on_exp_like_with_beta_power_series}
Let $\beta \in (0,1)$. 
Then there exists an $L>0$ such that for $z\ge 0$
\begin{align*}
  \sum_{k = 0}^{\infty} \frac{z^k}{(k!)^{\beta}}    \le L \exp (L z^{\frac{1}{\beta}}) .
\end{align*}
\end{lemma}

\begin{proof}
Let $\delta>0$. By writing $z^k = ((1+\delta)z)^k (1+\delta)^{-k}$ we get with H\"older's inequality 
  \begin{align*}
    \sum_{k = 0}^{\infty} \frac{z^k}{(k!)^{\beta}} 
& \le
    \left( \sum_{k = 0}^{\infty} \left( \frac{((1 + \delta)
    z)^k}{(k!)^{\beta}} \right)^{\frac{1}{\beta}} \right)^{\beta} 
    \left( \sum_{k =   0}^{\infty} (1 + \delta)^{- \frac{k}{1 - \beta} } \right)^{1 - \beta} \simeq \exp (\beta (1 + \delta)^{\frac{1}{\beta}} z^{\frac{1}{\beta}}).
  \end{align*}
\end{proof}

\begin{lemma}
\label{lemma:difference_gamma_and_p}
There exists a $C>0$ (independent of $b$) such that for all $\mu \in \N_0^d$ with $| \mu | \le 1$, and for all $t>0$, $x,y \in \R^d$, 
\begin{align}
\label{eqn:expression_partial_mu_Gamma}
&  \partial^\mu_x \Gamma_t(x,y) 
=
 \partial^\mu_x p(t,x-y)
 + \sum_{k = 1}^{\infty} \int_0^t \int_{\mathbb{R}^d} \partial^{\mu}_x p (t - s, x -  z) \Psi^{y,k}_{s,t} ( z ) \dd z \dd s,  \\
\notag 
&|\partial^\mu_x \Gamma_t(x,y) - \partial^\mu_x p(t,x-y)| \\
\notag 
&\qquad  \le  C   t^{- \frac{| \mu |}{2}}  p (c t, x - y)
 ( \| \Delta_{-1} b \|_{  C_t   L^\infty} t^{\frac12} \vee \| \Delta_{\ge 0} b \|_{  C_t   B_{\infty, 1}^{- \alpha}} t^{\frac{1-\alpha}{2}})
 \\
\label{eqn:difference_gamma_and_p}
& \qquad \qquad  \times 
  \exp \left( C t \Big[ \| \Delta_{-1} b \|_{  C_t   L^\infty}^{2}+ \| \Delta_{\ge 0} b \|_{  C_t   B_{\infty, 1}^{- \alpha}}^{\frac{2}{1 - \alpha}} \Big]\right)
 . 
\end{align}
\end{lemma}
\begin{proof}
 
To show both \eqref{eqn:expression_partial_mu_Gamma} and \eqref{eqn:difference_gamma_and_p} it is sufficient to estimate 
the series with the modulus of each term in the series in the right-hand side of \eqref{eqn:expression_partial_mu_Gamma} by the right-hand side of \eqref{eqn:difference_gamma_and_p}.

Let $K,$ $C,M$ be as in Lemma~\ref{lemma:bound_cIs}. 
Again, we will write ``$X$'' and ``$Y$'' instead of ``$\| \Delta_{-1} b \|_{  C_t   L^\infty}$'' and ``$ \| \Delta_{\ge 0} b \|_{  C_t   B_{\infty, 1}^{- \alpha}}$''. 
With $i = |\mu|$
\begin{align*}
& \sum_{k = 1}^{\infty} \int_0^t  \left| \int_{\mathbb{R}^d} \partial^\mu_x p (t - s, x - z) \Psi^{y,k}_{s,t} (z ) \dd z \right| \dd s 
\le \left( \sum_{k = 1}^{\infty} \cI_{i,k}^0(t) \right) p(ct,x-y) \\
& \qquad \le K t^{-\frac{i}{2}}  p (c t, x - y)
 \sum_{\substack{ m,n\in \N_0:\\ m+n \ge 1} } 
\frac{(CM X t^{\frac12} )^m }{( m!)^{\frac{1}{2} } }
\frac{(CM Y  t^{\frac{1-\alpha}{2} } )^n }{(n!)^{\frac{1-\alpha}{2} } }
  \\
& \qquad \le  K t^{-\frac{i}{2}}  p (c t, x - y)  C M  (Xt^{\frac12}   +   Y t^{\frac{1-\alpha}{2}}) \\
& \hspace{3cm} \times
\left( \sum_{ m\in \N_0 } 
\frac{(CM X t^{\frac12} )^m }{( m!)^{\frac{1}{2} } } \right) 
\left(
 \sum_{ n\in \N_0 } 
\frac{(CM Y  t^{\frac{1-\alpha}{2} } )^n }{(n!)^{\frac{1-\alpha}{2} } } 
\right) . 
\end{align*}
Indeed, for $a,b>0$ 
\begin{align*}
\sum_{\substack{ m,n\in \N_0:\\ m+n \ge 1} } 
\frac{a^m }{( m!)^{\frac{1}{2} } }
\frac{b^n }{(n!)^{\frac{1-\alpha}{2} } } 
& \le 
\sum_{  m,n\in \N_0 } 
\frac{a^{m+1} }{( (m+1)!)^{\frac{1}{2} } }
\frac{b^n }{(n!)^{\frac{1-\alpha}{2} } } 
+ 
\sum_{  m,n\in \N_0 } 
\frac{a^m }{( m!)^{\frac{1}{2} } }
\frac{b^{n+1} }{((n+1)!)^{\frac{1-\alpha}{2} } } \\
& 
\le (a+b) 
\sum_{  m,n\in \N_0 } 
\frac{a^{m} }{( m!)^{\frac{1}{2} } }
\frac{b^n }{(n!)^{\frac{1-\alpha}{2} } } .
\end{align*}
Now by applying Lemma~\ref{lemma:bound_on_exp_like_with_beta_power_series}
\begin{calc}
in both cases $z=  C M X t^{ \frac{1}{2}}$, $\beta = \frac{1}{2}$ and 
$z= C M Y  t^{ \frac{1 - \alpha}{2}}$, $\beta = \frac{1 - \alpha}{2}$,
\end{calc}
we obtain the desired bound. 
\end{proof}

\begin{proof}[Proof of the heat-kernel upper bound \eqref{eqn:heat_kernel_upper_bound} of Theorem~\ref{theorem:heat_kernel_bound}]
This is a direct consequence of Lemma~\ref{lemma:difference_gamma_and_p}, as there exists a $K>0$ such that for all $t\ge 0$ 
\begin{align*}
C t (Xt^{\frac12} \vee Y t^{\frac{1-\alpha}{2}}) \le 
  \exp \left( K t [ X^{2}+  Y^{\frac{2}{1 - \alpha}} ]\right) .
\end{align*}
\end{proof}

\section{Heat-kernel lower bounds}
\label{section:heat_kernel_lower_bound_for_smooth_drift}

The lower bound follows from Lemma~\ref{lemma:difference_gamma_and_p} together with the next result, which is a small variation of \cite[Lemma 4.3.8]{St08}.

\begin{lemma}
\label{lemma:bound_on_kernels_that_satisfy_CK}
Let $q_t: \R^d \times \R^d \rightarrow [0,\infty)$ for all $t\in [0,\infty)$. 
Suppose that $(q_t)_{t\in [0,\infty)}$ satisfies the Chapman-Kolmogorov equations, i.e., $q_{t+s}(x,y) = \int_{\R^d} q_t(x,z) q_s(z,y) \dd z$. 
Let $a,b>0$. 
Suppose that $q_t(x,y) \ge b t^{-\frac{d}{2}}$ for all $t\in (0,a]$ and $x,y \in \R^d$ with $|x-y|\le \sqrt{t}$. 
Then there exist a $\kappa \in (0,1)$ and an $M>1$, which only depends on $b$ and $d$, such that for all $t\in [0,\infty)$ and $x,y \in \R^d$
\begin{align*}
q_t(x,y) \ge M^{-1-\frac{t}{a}}  p(\kappa t, x-y). 
\end{align*}
\end{lemma}
\begin{proof}
By following the first step of the proof of \cite[Lemma 4.3.8]{St08} we find a $\kappa \in (0,1)$ and a $M>1$ which depend only on $b$ and $d$ such that for all $t\in (0,a]$ and $x,y \in \R^d$ 
\begin{align*}
q_t(x,y) \ge
 M^{-1} p(\kappa t, x-y). 
\end{align*}
Let $t>a$ and $n = \lceil \frac{t}{a} \rceil$. Then for all $x,y \in R^d$
\begin{align*}
q_t(x,y) 
& = \int_{(\R^d)^{n-1}} q_{\frac{t}{n}}(x,z_1) q_{\frac{t}{n}}(z_1,z_2) \cdots q_{\frac{t}{n}}(z_{n-1},y) \dd z \\
& \ge 
\int_{(\R^d)^{n-1}}  M^{-n} p (\kappa \tfrac{t}{n}, x -z_1) p (\kappa \tfrac{t}{n}, z_1 -z_2)\cdots p (\kappa \tfrac{t}{n}, z_{n-1} -y ) \dd z \\
& \ge M^{-1-\frac{t}{a}} p(\kappa t, x-y). 
\end{align*}
\end{proof}

Now we can prove the heat-kernel lower bounds: 

\begin{proof}[Proof of the heat-kernel lower bound \eqref{eqn:heat_kernel_lower_bound} of Theorem~\ref{theorem:heat_kernel_bound}]
We want to apply Lemma~\ref{lemma:bound_on_kernels_that_satisfy_CK}. 
Therefore we will find an $a$ such that the condition is satisfied.
Once more we will write ``$X$'' and ``$Y$'' instead of ``$\| \Delta_{-1} b \|_{  C_t   L^\infty}$'' and ``$ \| \Delta_{\ge 0} b \|_{  C_t   B_{\infty, 1}^{- \alpha}}$''. 
Let us also take $X=\| \Delta_{-1} b \|_{  C_t    L^\infty}$ and $Y=  \| \Delta_{\ge 0} b \|_{  C_t   B_{\infty, 1}^{- \alpha}}$. 
Let $\alpha \in (0,\frac12)$, $c > 1$ and $C>0$ be as in Lemma~\ref{lemma:difference_gamma_and_p}. 
Then~\eqref{eqn:difference_gamma_and_p} gives for $a>0$, $t\in (0,a]$ and $x,y \in \R^d$ with $|x-y| \le \sqrt{t}$: 
\begin{align*}
& \Gamma_t(x,y) 
 \ge p(t,x-y) - C  (Xt^{\frac12} \vee Y t^{\frac{1-\alpha}{2}})
  \exp \left( C t [  X^{2}+  Y^{\frac{2}{1 - \alpha}} ]\right) 
    p (c t, x - y) \\
	& \ge (2\pi t)^{-\frac{d}{2}} e^{-\frac12} - C  ( (X^2a)^{\frac12} \vee (Y^{\frac{2}{1-\alpha}} a)^{\frac{1-\alpha}{2}})  \exp \left( C a [X^2 + Y^{\frac{2}{1 - \alpha}}] \right) c^{-\frac{d}{2}} (2\pi t)^{-\frac{d}{2}}.
\end{align*}
Therefore,  it holds that $\Gamma_t(x,y) \ge \tfrac12 (2\pi t)^{-\frac{d}{2}} e^{-\frac12}$ if
\begin{align*}
C ( (X^2a)^{\frac12} \vee (Y^{\frac{2}{1-\alpha}} a)^{\frac{1-\alpha}{2}})  \exp \left( C a [X^2 + Y^{\frac{2}{1 - \alpha}}] \right) c^{-\frac{d}{2}} \le \frac{e^{-\frac12}}{2}. 
\end{align*}
Hence there exists a $K\in (0,1)$ (which only depends on $c,C$ and $\alpha$)  
such that the choice $a = K [X^2 + Y^{\frac{2}{1 - \alpha}}]^{-1}$  works.
\begin{calc}
For this $a$ we have (as $K\in (0,1)$ we use that $K \le K^{1-\alpha}$)
\begin{align*}
& C ( (X^2a)^{\frac12} \vee (Y^{\frac{2}{1-\alpha}} a)^{\frac{1-\alpha}{2}})  \exp \left( C a [X^2 + Y^{\frac{2}{1 - \alpha}}] \right) c^{-\frac{d}{2}} \\
& \le 
C  K^{\frac{1-\alpha}{2}}  \exp \left( C K \right) c^{-\frac{d}{2}}. 
\end{align*}
\end{calc}
So by Lemma~\ref{lemma:bound_on_kernels_that_satisfy_CK} there exist a $\kappa \in (0,1)$ and a $M > 1$ such that for all $t\in [0,\infty)$ and $x,y\in \R^d$, 
\begin{align*}
\Gamma_t(x,y) 
&\ge 
 M^{-1-\frac{t}{a}}  p(\kappa t, x-y) = \frac1M
\exp\bigg(- t \frac{\log M}{K  } [X^2 + Y^{\frac{2}{1 - \alpha}}]  \bigg)  p(\kappa t, x-y). 
\end{align*}
This proves that \eqref{eqn:heat_kernel_lower_bound} holds for a large enough $C$. 
\end{proof}

\section{Proof of Corollary~\ref{corollary:escape_out_of_box_probability}}
\label{section:probability_of_escaping_the_box}

As before, we consider $b\in C([0,T],B_{\infty,1}^{-\alpha})$ for some $\alpha \in (0,\frac12)$ and we let $X= (X_t)_{t\in [0,T]}$ be the  solution to the martingale problem for $(  (\cL_t)_{t \in (0,T]}  , \delta_x)$. 
We prove Corollary~\ref{corollary:escape_out_of_box_probability}, which means that we  estimate the probability that $X$ escapes a box of size $K$ before time $T$. 
The estimate is a consequence of our heat-kernel estimates (Theorem~\ref{theorem:heat_kernel_bound}), Markov's inequality and the Garsia-Rademich-Rumsey inequality. 
By the latter (see \begin{calc}
  Theorem~\ref{theorem:garsia-rademich-rumsey} or
  \end{calc} \cite[Theorem 2.1.3]{StVaSr06}) we have for $\kappa>0$
  \begin{align} 
  \label{eqn:difference_X_t_X_s_est_by_garsia_R_R}
 \kappa | X_t - X_s | \le 4 \int_0^{t - s} u^{- \frac{1}{2}} \sqrt{\log
     \left( 1 + \frac{4 (F_{T,\kappa} - T^2)}{u^2} \right)} \dd u, 
  \end{align}
where 
  \begin{align}
\label{eqn:F_integral_def}
   F_{T,\kappa} = \int_0^T \int_0^T \exp \bigg( \kappa \bigg( \frac{|
     X_{r_2} - X_{r_1} |}{| r_2 - r_1 |^{\frac{1}{2}}} \bigg)^2 \bigg)
     \dd r_1 \dd r_2 .
 \end{align}

\begin{calc}
Indeed, with $\Psi(t) = e^{t^2} -1$ (so that $\Psi^{-1}(s) = \sqrt{\log(1+s)}$), $p(t) =  \sqrt{t}$ (so that $p(\dd u) = \frac{1}{2} u^{-\frac12}$) and $\phi = \kappa X$, 
\begin{align*}
B:= \int_0^T \int_0^T \Psi \left( \frac{|\phi(t) -\phi(s)|}{p(|t-s|)}\right)  \dd s \dd t
= F_{T,\kappa} - T^2 ,
\end{align*}
so that 
\begin{align*}
\kappa | X_t - X_s | \le 
 8 \int_0^{t-s} \Psi^{-1} \left(\frac{4}{u^2} B \right)  p( \dd u) 
\end{align*}
gives \eqref{eqn:difference_X_t_X_s_est_by_garsia_R_R} by Theorem~\ref{theorem:garsia-rademich-rumsey}. 
\end{calc}

\begin{calc}
\begin{theorem}[Garsia-Rademich-Rumsey inequality] \cite[Theorem 2.1.3]{StVaSr06}
\label{theorem:garsia-rademich-rumsey}
Let $p$ and $\Psi$ be continuous and strictly continuous functions on $[0,\infty)$ such that 
\begin{align*}
p(0)=\Psi(0)=0, \qquad \lim_{t\rightarrow \infty} \Psi(t) = \infty.
\end{align*}
Let $T>0$ and $\phi \in C([0,T], \R^d)$. 
Then for $0\le s < t \le T$ 
\begin{align*}
|\phi(t) - \phi(s) | \le 8 \int_0^{t-s} \Psi^{-1} \left(\frac{4}{u^2} \int_0^T \int_0^T \Psi \left( \frac{|\phi(t) -\phi(s)|}{p(|t-s|)}\right)  \dd s \dd t \right) p( \dd u) .
\end{align*}
\end{theorem}
\end{calc}

In the proof of Corollary~\ref{corollary:expectation_of_exponential_of_sup} we will bound the right-hand side of \eqref{eqn:difference_X_t_X_s_est_by_garsia_R_R} in terms of a function $\zeta$. 
In the next lemma we start by gathering some auxiliary facts about $\zeta$. 
\begin{lemma}
\label{lemma:about_zeta_psi}
Let $\zeta, \psi \colon (0,\infty) \rightarrow (0,\infty)$ be given by
  \begin{align*} 
  \zeta (r) := \int_0^r u^{- \frac{1}{2}} \left( \sqrt{\log (1 + u^{- 2})}
     \vee 1 \right) \dd u, \qquad \psi (r) := r^{\frac{1}{2}} 
     \sqrt{(\log (\tfrac{1}{r}) \vee 1)}  . 
     \end{align*}
  There exist $m, M > 0$ such that $m \zeta (r) \le \psi (r) \le M \zeta (r) $ for all $r>0$. 
  Moreover, $\psi (r s) \le \sqrt{2} \psi (r) \psi (s)$ for all $r, s >0$ and $\psi$ is strictly increasing. 
\end{lemma}
\begin{proof}
That $\psi$ is strictly increasing on $(\frac{1}{e},\infty)$ will be clear, whereas on $[0,\frac{1}{e})$ it follows by calculating its derivative. 
Since $\psi $ and $\zeta $ are continuous and bounded away from $0$ and $\infty$ on compact subintervals of $(0, \infty)$, the existence of such $m$ and $M$ follows once we show that $\lim_{r\rightarrow 0} \frac{\zeta (r)}{\psi (r)}$ and $\lim_{r\rightarrow \infty} \frac{\zeta (r)}{\psi (r)}$ exist and are in $(0,\infty)$.  By applying L'Hospital's rule we obtain 
  \begin{align*}
   \lim_{r \rightarrow 0} \frac{\zeta (r)}{\psi (r)} = \lim_{r \rightarrow
     0} \frac{\int_0^r u^{- \frac{1}{2}} \sqrt{\log (1 + u^{- 2})} \dd u}{r^{\frac{1}{2}} \sqrt{\log (\tfrac{1}{r})}} \in (0, \infty).
      \end{align*}
\begin{calc}
Indeed
\begin{align*}
&  \lim_{r\rightarrow 0}  
 \frac{\int_0^r u^{- \frac{1}{2}} \sqrt{\log (1 + u^{- 2})} \dd u}{r^{\frac12} \sqrt{\log (\tfrac{1}{r})}}  \\
&  =  \lim_{r \rightarrow 0} \frac{ r^{- \frac{1}{2}} \sqrt{\log (1 + r^{- 2})} }{\frac12 r^{-\frac12} \sqrt{\log (\tfrac{1}{r})} - \frac12 r^{\frac{1}{2}} \log (\tfrac{1}{r})^{-\frac12}  r^{-1} } 
 =  \lim_{r \rightarrow 0} \frac{  \sqrt{\log (1 + r^{- 2})} }{\frac12 \sqrt{\log (\tfrac{1}{r})} - \frac12  \log (\tfrac{1}{r})^{-\frac12}  } , \\
 & =  \lim_{a \rightarrow \infty} 2\frac{  \sqrt{\log (1 + a^2)} }{ \sqrt{\log a} -   \log (a)^{-\frac12}  } 
 =  2 \sqrt{  \lim_{a \rightarrow \infty} \frac{  \log (1 + a^2) }{ \log a } }
  =  2 \sqrt{  \lim_{a \rightarrow \infty} \frac{2a^2}{ 1 + a^2 } } = 2 \sqrt 2. 
\end{align*}
\end{calc}     
And also for $r\to \infty$ we have
  \begin{align*} 
  \lim_{r \rightarrow \infty} \frac{\zeta (r)}{\psi (r)} 
& = \lim_{r      \rightarrow \infty} \frac{\int_0^{\sqrt{e - 1}} u^{- \frac{1}{2}} \sqrt{\log (1 + u^{- 2})} \dd u + \int_{\sqrt{e - 1}}^r u^{- \frac{1}{2}} \dd u}{r^{\frac12}} 
\cnewline \cand \begin{calc}
= \lim_{r      \rightarrow \infty} r^{-\frac12} \int_{\sqrt{e - 1}}^r u^{- \frac{1}{2}} \dd u 
= \lim_{r      \rightarrow \infty} r^{-\frac12} 2[\sqrt{r} - \sqrt{\sqrt{e-1}} ] =2
\end{calc}
 \in (0, \infty) . 
     \end{align*}
Furthermore
  \begin{align*}
   \psi (r s) = (r s)^{\frac{1}{2}} \left( \sqrt{(\log (\tfrac{1}{r}) + \log (\tfrac{1}{s}))
     \vee 1} \right) 
     \end{align*}
  and for all $x, y \in \mathbb{R}$ we have $(x + y) \vee 1 \le x \vee 1 + y \vee 1 \le 2 (x
  \vee 1) (y \vee 1)$. 
  Therefore,
  \begin{align*} 
  \psi (r s) 
  \le \sqrt 2 (r s)^{\frac{1}{2}} \left( \sqrt{\log (\tfrac{1}{r}) \vee 1}
     \right) \left( \sqrt{\log (\tfrac{1}{s}) \vee 1} \right) 
   = \sqrt 2 \psi (r) \psi (s) .
  \end{align*}
\end{proof}

\begin{corollary}
\label{corollary:expectation_of_exponential_of_sup}
Let $\psi$ be as in Lemma~\ref{lemma:about_zeta_psi}
and let $C>0$ be as in Theorem~\ref{theorem:heat_kernel_bound}. 
Then there exists an $M > 0$ such that for all $T \ge 1$ \tonote{in previous versions we put another factor in front of the exp, like $T^2 \wedge \|b\|^{-\frac{4}{1-\alpha}}$} 
  \begin{align} 
\notag 
&  \mathbb{E}_x \bigg[ \exp \bigg( \frac{1}{M} \bigg( \sup_{ \substack{s,t\in [0,T] \\ s<t }} \frac{| X_t - X_s |}{\psi (t - s)} \bigg)^2  \bigg) \bigg] \\
\label{eqn:expectation_of_exponential_of_sup}
&   \le M \exp \Big( C T \Big[  \| \Delta_{-1} b \|_{  C_T   L^\infty}^{2}  + \| \Delta_{\ge 0} b \|_{  C_T   B_{\infty, 1}^{- \alpha}}^{\frac{2}{1 - \alpha}}  \Big] \Big). 
 \end{align}
\end{corollary}
\begin{proof}
The proof is inspired by \cite[Corollary A.5]{FrVi10}. Unfortunately we cannot directly apply that result, because the constant they derive depends on the time interval $[0,T]$ (even though this is not explicitly stated).

Let us define $G_{T,\kappa}: = 2 \sqrt{F_{T,\kappa} \vee 4}$, where $F_{T,\kappa}$ is as in \eqref{eqn:F_integral_def}. 
Let $\zeta$ be as in Lemma~\ref{lemma:about_zeta_psi}. 
By \eqref{eqn:difference_X_t_X_s_est_by_garsia_R_R} and using $4(F_{T,\kappa}-T^2)\le G_{T,\kappa}^2$ we have by a substitution and by Lemma~\ref{lemma:about_zeta_psi} (observe that $G_{T,\kappa} \ge 4 \ge e$) that for $T\ge 1$, $\kappa>0$, $s,t\in [0,T]$ with $s<t$ and by writing $G= G_{T,\kappa}$
\begin{align*}
\kappa  | X_t - X_s | 
  \cand
  \begin{calc} \le 4 \int_0^{t - s} u^{- \frac{1}{2}} \sqrt{\log
     \left( 1 + \frac{G^2}{u^2} \right)} \dd u \end{calc} \cnewline 
   & \le 4 \sqrt{G} \int_0^{\frac{t-s}{G}} u^{- \frac{1}{2}} \sqrt{\log
     \left( 1 + \tfrac{1}{u^2} \right)} \dd u 
   \lesssim \sqrt{G} \zeta( \tfrac{t-s}{G}) \\
& \lesssim \sqrt{G} \psi( \tfrac{t-s}{G}) 
 \lesssim \sqrt{G} \psi( t-s) \psi(\tfrac{1}{G})
 \lesssim \psi( t-s)  \sqrt{\log G}. 
\end{align*}
Let $M>0$ be such that $ \kappa  | X_t - X_s |  \le \sqrt{M} \psi( t-s)  \sqrt{\log G_{T,\kappa}}$ for all $T\ge 1$, $\kappa >0$ and $s,t\in [0,T]$ with $s<t$. 
Then 
\begin{align*}
\E_x \left[  \exp \left( \frac{\kappa^2}{M} \left( \sup_{ \substack{s,t\in [0,T] \\ s<t }}
     \frac{| X_t - X_s |}{\psi (t - s)} \right)^2 \right) \right]
\le \E_x[ G_{T,\kappa}]. 
\end{align*}
As by Jensen's inequality $\E_x[G_{T,\kappa}] = 2\E_x[\sqrt{F_{T,\kappa}\vee 4}] \le 2 \sqrt{\E_x[F_{T,\kappa}]+4}$ we will obtain a bound of $\E_x[G_{T,\kappa}]$, by estimating $\E_x[F_{T,\kappa}]$.
Let $c \in (0,1)$ and $\kappa>0$ be such that $\kappa < \frac{1}{2c}$. Then for all $r_2, r_1>0$ with $r_2 \ne r_1$
  \begin{align}
  \label{eqn:gaussian_calculation}
  \int_{\mathbb{R}^d} p(c|r_2 - r_1|, y )\exp (  \kappa \Big( \frac{| y  |}{| r_2 - r_1 |^{\frac{1}{2}}} \Big)^2 ) \dd y =(\tfrac{1}{1-2c\kappa  })^{\frac{d}{2}} < \infty.
  \end{align}
\begin{calc}
\begin{align*}
& \int_{\R^d} p(c t, y) \exp( \kappa \frac{|y|^2}{t} ) \dd y  
= (2\pi ct)^{-\frac{d}{2}} \int_{\R^d}  \exp \left( (\kappa - \frac{1}{2c}  ) \frac{|y|^2}{t} \right) \dd y \\
& = (2\pi (\tfrac{1}{\frac{1}{c}- 2\kappa}) t)^{-\frac{d}{2}} 
(\tfrac{1}{c(\frac{1}{c}- 2\kappa) })^{\frac{d}{2}} 
\int_{\R^d}  \exp \left( - \frac{|y|^2}{2 (\frac{1}{ \frac{1}{c}- 2\kappa }) t} \right) \dd y 
=(\tfrac{1}{1-2c\kappa  })^{\frac{d}{2}} .
\end{align*}
\end{calc}
Hence, by Theorem~\ref{theorem:heat_kernel_bound} 
  \begin{align*}
    \mathbb{E}_x [F_{T,\kappa}] 
   \cand \begin{calc}
    = \int_0^T \int_0^T \mathbb{E}_x \left[ \exp
    ( \kappa  \left( \frac{| X_{r_2} - X_{r_1} |}{| r_2 - r_1  |^{\frac{1}{2}}} \right)^2 ) \right] \dd r_1 \dd r_2
    \end{calc} \cnewline
    & = \int_0^T \int_0^T \mathbb{E}_x \left[ \int_{\mathbb{R}^d} \Gamma_{|r_2 - r_1|}( y ,  X_{r_1})\exp (  \kappa \left( \frac{| y -     X_{r_1} |}{| r_2 - r_1 |^{\frac{1}{2}}} \right)^2 ) \dd y \right]
    \dd r_1 \dd r_2\\
    & \le C (\tfrac{1}{1-2c\kappa  })^{\frac{d}{2}}  \int_0^T \int_0^T \exp \left( C |r_2 - r_1|  \Big[  \| \Delta_{-1} b \|_{  C_t   L^\infty}^{2}  + \| \Delta_{\ge 0} b \|_{  C_t   B_{\infty, 1}^{- \alpha}}^{\frac{2}{1 - \alpha}}  \Big]  \right) \dd r_1
    \dd r_2 .
  \end{align*}
The proof is completed by observing that for $A\ge 1$
\begin{align*}
\int_0^T \int_0^T \exp \left( A |r_2 - r_1| \right) \dd r_1
    \dd r_2  
= 2  \int_0^T \int_0^{t} e^{A (t - s)} \dd s \dd t 
\lesssim 
\begin{calc} \frac{2}{A} \int_0^{T} e^{A t} \dd t 
\lesssim  \frac{e^{AT}}{A^2}  \le \end{calc}
e^{AT}.
\end{align*}
\end{proof}

\begin{proof}[Proof of Corollary \ref{corollary:escape_out_of_box_probability}]
As $T \ge 1\ge e^{- 1}$ we have $\psi
  (T) = \sqrt{T}$. 
Therefore, by Markov's inequality for all $M,K>0$ and the fact that $\psi$ is strictly increasing:
\begin{align*}
    \mathbb{P}_x \Big(\sup_{t\in[0,T]} | X_t - x | \ge K\Big) 
& \le \mathbb{E}_x \Big[ \exp \Big( \frac{1}{M T}  \sup_{t\in[0,T]} | X_t - x |^2 \Big) \Big] \exp  \Big( - \frac{K^2}{M T} \Big)\\
& \le \mathbb{E}_x \Big[ \exp \Big( \frac{1}{M} \Big(\sup_{ \substack{s,t\in [0,T] \\ s<t }} \frac{| X_t - X_s |}{\psi (t - s)}\Big)^2 \Big) \Big] \exp \Big( - \frac{K^2}{M T} \Big).
\end{align*}
So \eqref{eqn:escape_out_of_box_probability} follows from Corollary~\ref{corollary:expectation_of_exponential_of_sup}.
\end{proof}

\appendix

\section{Appendix}


\begin{theorem}
\label{theorem:product_estimates}
Suppose $\alpha <0$ and $\beta >0$ are such that $\alpha + \beta >0$. 
Let $p,p_1,p_2, q_1,q_2 \in [1,\infty]$ be such that 
\begin{align}
\label{eqn:holder_coefficients}
\tfrac1p = \tfrac{1}{p_1} + \tfrac{1}{p_2} .
\end{align}
For all $r\ge q_1$
\begin{align}
\label{eqn:product_in_B_p_r_no_loss}
\| u \cdot v \|_{B_{p,r}^{\alpha } } 
& \lesssim  \| u \|_{B_{p_1,q_1}^\alpha} \| v \|_{B_{p_2,q_2}^\beta }. 
\end{align}
\end{theorem}
\begin{proof}
For the proof see also \cite[Corollary 2.1.35]{Ma18}.
By slightly adapting \cite[Theorem 2.82]{BaChDa11} and by using the H\"older inequality and \cite[Theorem 2.79]{BaChDa11} (for \eqref{eqn:bound_arap_with_L_p}), we obtain implies the following two estimates. 
\begin{align}
\label{eqn:bound_para_plus_reso}
& \| u \para v \|_{B_{p,q}^{\alpha+\beta} } 
 \lesssim 
 \|u\|_{B_{p_1,q_1}^\alpha} \|v\|_{B_{p_2,q_2}^\beta}, \\
\label{eqn:bound_arap_with_L_p}
& \| u \arap v \|_{B_{p,r}^\alpha} 
 \lesssim 
\| v \|_{L^{p_2}} \| u \|_{B_{p_1,r}^\alpha}
\lesssim 
\| v \|_{B_{p_2,q_2}^\beta } \| u \|_{B_{p_1,q_1}^\alpha}. 
\end{align}
As \cite[Theorem 2.52]{BaChDa11} implies $ \| u \reso v \|_{B_{p,q}^{\alpha+\beta} } 
 \lesssim 
 \|u\|_{B_{p_1,q_1}^\alpha} \|v\|_{B_{p_2,q_2}^\beta}$, combining the above inequalities proves \eqref{eqn:product_in_B_p_r_no_loss}. 
\end{proof}

\begin{calc}
\begin{lemma}
\label{lemma:derivative_of_integral_parameter_and_integrand_same}
Let $\fX$ be a Banach space and $f : [0,\infty) \times [0,\infty) \rightarrow \fX$ be continuously differentiable and be such that $s \mapsto f(t-s,s)$ is Bochner integrable on $[0,t]$. 
Define $F(t) := \int_0^t f(t-s,s) \dd s$. 
Then $F$ is differentiable and 
\begin{align*}
\partial_t F(t) = f(0,t) + \int_0^t  \rD_1 f(t-s,s) \dd s. 
\end{align*}
\end{lemma}
\begin{proof}
$F(t) = G(t,t)$, where $G(t,u) =\int_0^t f(u-s,s) \dd s$. By the Lebesgue dominated convergence theorem for Bochner integrals, $\partial_u G(t,u) =  \int_0^t  \rD_1 f(u-s,s) \dd s$. The rest follows by the Fundamental law of Calculus and the chain rule (one may also want to use that Bochner integrable functions are Pettis integrable and their integrals agree). 
\end{proof}
\end{calc}

\textbf{Acknowledgements.} 
This work was supported by the German Science Foundation (DFG) via the Forschergruppe FOR2402 ``Rough paths, stochastic partial differential equations and related topics''. WvZ was supported by the DFG through SPP1590 ``Probabilistic Structures in Evolution''. 
NP thanks the DFG for financial support through the Heisenberg program.
The main part of the work was done while NP was employed at Humboldt-Universit\"at zu Berlin and Max-Planck-Institute for Mathematics in the Sciences, Leipzig. The
authors are also grateful to the anonymous referees for their valuable feedback, suggestions and careful
reading.

\bibliographystyle{abbrv}
\bibliography{references}

\end{document}

%% file: style_definitions.tex
\usepackage{amsmath,amstext,amscd,amssymb,euscript,mathrsfs}
\usepackage{calrsfs}
\usepackage{epsfig}
\usepackage{mathtools}
\usepackage[colorlinks=false]{hyperref} 
\usepackage[english]{babel}
\usepackage{verbatim} 
\usepackage{enumerate}
\usepackage{bbm}
\usepackage{color}
\usepackage{wrapfig} 

\newcommand{\SortNoop}[1]{}

\usepackage{enumitem}
\setenumerate{label={\normalfont(\alph*)},topsep=4pt,itemsep=0pt} 

%% file: theorem_definitions.tex
\theoremstyle{plain}
\newtheorem{theorem}{Theorem}[section]
\newtheorem{lemma}[theorem]{Lemma}
\newtheorem{proposition}[theorem]{Proposition}
\newtheorem{corollary}[theorem]{Corollary}

\theoremstyle{definition}
\newtheorem{definition}[theorem]{Definition}

\newtheorem{remark}[theorem]{Remark}



%% file: math_definitions.tex
\newcommand{\1}{\mathbbm{1}}

\newcommand{\dd}{\, \mathrm{d}}

\newcommand{\E}{\mathbb E}

\newcommand{\R}{\mathbb R}
\newcommand{\N}{\mathbb N}

\newcommand{\cC}{\mathcal{C}}

\newcommand{\cF}{\mathcal{F}}

\newcommand{\cI}{\mathcal{I}}

\newcommand{\cL}{\mathcal{L}}

\newcommand{\cS}{\mathcal{S}}

\newcommand{\fX}{\mathfrak{X}}

\renewcommand{\P}{\mathbb{P}}

\newcommand{\supp}{\operatorname{supp}} 

\usepackage{mathtools} 

\newcommand{\rD}{\mathrm{D}}

\newcommand{\para}{\varolessthan}
\newcommand{\arap}{\varogreaterthan}
\newcommand{\reso}{\varodot}

\newcommand{\I}{\mathbb{I}}

\renewcommand{\epsilon}{\varepsilon}

%% file: writing_definitions.tex
\usepackage[normalem]{ulem} 

\newenvironment{colorwil}
    {\color[rgb]{0,0.5,0}    }

\newcommand{\bw}{\bgroup\color[rgb]{0,0.5,0}}
\newcommand{\ew}{\egroup}

\newcommand{\bwil}{\begin{colorwil}}
\newcommand{\ewil}{\end{colorwil}}

\newcommand{\tonote}{\todo[color=gray!50]}

\newenvironment{calc}    {\color{orange}     }    {     }

\newcommand{\cnewline}{\\}
\newcommand{\cand}{&}

\usepackage{environ}


%% file: title_arxiv_hkb.tex
\renewcommand{\thefootnote}{\Roman{footnote}}

\title{Quantitative heat-kernel estimates for diffusions with distributional drift}

\author{
\renewcommand{\thefootnote}{\Roman{footnote}}
Nicolas Perkowski
\footnotemark[1]
\\
\renewcommand{\thefootnote}{\Roman{footnote}}
Willem van Zuijlen
\footnotemark[2]
}

\footnotetext[1]{Free University of Berlin, Arnimallee 7, 14195 Berlin, Germany {\tt perkowski@math.fu-berlin.de}
}
\footnotetext[2]{Weierstrass Institute for Applied Analysis and Stochastics, Mohrenstra{\ss}e 39, 10117 Berlin, Germany, {\tt vanzuijlen@wias-berlin.de}}

\date{January 27, 2022}

\maketitle 

\renewcommand{\thefootnote}{\arabic{footnote}} 


\begin{abstract}
We consider the stochastic differential equation on $\R^d$ given by 
\begin{align*}
\dd X_t = b(t,X_t) \dd t + \dd B_t, 
\end{align*}
where $B$ is a Brownian motion and $b$ is considered to be a distribution of regularity $ > -\frac12$. 
We show that the martingale solution of the SDE has a transition kernel $\Gamma_t$ and prove upper and lower heat-kernel estimates for $\Gamma_t$ with explicit dependence on $t$ and the norm of $b$. 

\bigskip

\emph{Keywords and phrases.} heat-kernel estimate, singular diffusion, parametrix method. 

\emph{MSC 2020}. {\em Primary.} 60H10 {\em Secondary.} 35A08.



\end{abstract}

%% file: heat_kernel_bounds.bbl
\begin{thebibliography}{10}

\bibitem{AnAsRo99}
G.~E. Andrews, R.~Askey, and R.~Roy.
\newblock {\em Special functions}, volume~71 of {\em Encyclopedia of
  Mathematics and its Applications}.
\newblock Cambridge University Press, Cambridge, 1999.

\bibitem{BaChDa11}
H.~Bahouri, J.-Y. Chemin, and R.~Danchin.
\newblock {\em Fourier analysis and nonlinear partial differential equations},
  volume 343 of {\em Grundlehren der Mathematischen Wissenschaften [Fundamental
  Principles of Mathematical Sciences]}.
\newblock Springer, Heidelberg, 2011.

\bibitem{BaCh01}
R.~F. Bass and Z.-Q. Chen.
\newblock Stochastic differential equations for {D}irichlet processes.
\newblock {\em Probab. Theory Related Fields}, 121(3):422--446, 2001.

\bibitem{Br86}
T.~Brox.
\newblock A one-dimensional diffusion process in a {W}iener medium.
\newblock {\em Ann. Probab.}, 14(4):1206--1218, 1986.

\bibitem{CaCh18}
G.~Cannizzaro and K.~Chouk.
\newblock Multidimensional {SDE}s with singular drift and universal
  construction of the polymer measure with white noise potential.
\newblock {\em Ann. Probab.}, 46(3):1710--1763, 2018.

\bibitem{DeDi16}
F.~Delarue and R.~Diel.
\newblock Rough paths and 1d {SDE} with a time dependent distributional drift:
  application to polymers.
\newblock {\em Probab. Theory Related Fields}, 165(1-2):1--63, 2016.

\bibitem{FlIsRu17}
F.~Flandoli, E.~Issoglio, and F.~Russo.
\newblock Multidimensional stochastic differential equations with
  distributional drift.
\newblock {\em Trans. Amer. Math. Soc.}, 369(3):1665--1688, 2017.

\bibitem{FlRuWo03}
F.~Flandoli, F.~Russo, and J.~Wolf.
\newblock Some {SDE}s with distributional drift. {I}. {G}eneral calculus.
\newblock {\em Osaka J. Math.}, 40(2):493--542, 2003.

\bibitem{Fr64}
A.~Friedman.
\newblock {\em Partial differential equations of parabolic type}.
\newblock Prentice-Hall, Inc., Englewood Cliffs, N.J., 1964.

\bibitem{Fr75}
A.~Friedman.
\newblock {\em Stochastic differential equations and applications. {V}ol. 1}.
\newblock Academic Press [Harcourt Brace Jovanovich, Publishers], New
  York-London, 1975.
\newblock Probability and Mathematical Statistics, Vol. 28.

\bibitem{FrVi10}
P.~K. Friz and N.~B. Victoir.
\newblock {\em Multidimensional stochastic processes as rough paths}, volume
  120 of {\em Cambridge Studies in Advanced Mathematics}.
\newblock Cambridge University Press, Cambridge, 2010.
\newblock Theory and applications.

\bibitem{GuImPe15}
M.~Gubinelli, P.~Imkeller, and N.~Perkowski.
\newblock Paracontrolled distributions and singular {PDE}s.
\newblock {\em Forum Math. Pi}, 3:e6, 75, 2015.

\bibitem{GuPe17}
M.~Gubinelli and N.~Perkowski.
\newblock K{PZ} reloaded.
\newblock {\em Comm. Math. Phys.}, 349(1):165--269, 2017.

\bibitem{Ha11}
M.~Hairer.
\newblock {Rough stochastic {PDE}s}.
\newblock {\em Comm. Pure Appl. Math.}, 64(11):1547--1585, 2011.

\bibitem{Ha13}
M.~Hairer.
\newblock {Solving the {KPZ} equation}.
\newblock {\em Ann. Math.}, 178(2):559--664, 2013.

\bibitem{Ko16}
W.~K\"onig.
\newblock {\em The parabolic {A}nderson model}.
\newblock Pathways in Mathematics. Birkh\"auser/Springer, [Cham], 2016.
\newblock Random walk in random potential.

\bibitem{KoPevZ}
W.~K\"{o}nig, N.~Perkowski, and W.~B. van Zuijlen.
\newblock {Long-time asymptotics of the two-dimensional parabolic Anderson
  model with white-noise potential}.
\newblock Preprint available at \url{http://arxiv.org/abs/2009.11611}.

\bibitem{Ma18}
J.~Martin.
\newblock {\em Refinements of the Solution Theory for Singular SPDEs}.
\newblock PhD thesis, Humboldt-Universit\"at zu Berlin, 2018.

\bibitem{MaPe19}
J.~Martin and N.~Perkowski.
\newblock Paracontrolled distributions on {B}ravais lattices and weak
  universality of the 2d parabolic {A}nderson model.
\newblock {\em Ann. Inst. Henri Poincar\'{e} Probab. Stat.}, 55(4):2058--2110,
  2019.

\bibitem{St08}
D.~W. Stroock.
\newblock {\em Partial differential equations for probabilists}, volume 112 of
  {\em Cambridge Studies in Advanced Mathematics}.
\newblock Cambridge University Press, Cambridge, 2008.

\bibitem{StVaSr06}
D.~W. Stroock and S.~R.~S. Varadhan.
\newblock {\em Multidimensional diffusion processes}.
\newblock Classics in Mathematics. Springer-Verlag, Berlin, 2006.
\newblock Reprint of the 1997 edition.

\bibitem{ZhZh17}
X.~Zhang and G.~Zhao.
\newblock Heat kernel and ergodicity of {SDE}s with distributional drifts.
\newblock {\em arXiv preprint arXiv:1710.10537}, 2017.

\end{thebibliography}
